\documentclass[11pt,reqno]{amsart}
\usepackage{amssymb}
\usepackage{amsmath}
\usepackage{amsthm}
\usepackage{graphicx}

\usepackage{subfigure}

\usepackage{caption}
\usepackage{bm}
\usepackage[colorlinks=true,urlcolor=blue,
citecolor=red,linkcolor=blue,linktocpage,pdfpagelabels,
bookmarksnumbered,bookmarksopen]{hyperref}
\usepackage[titletoc,title]{appendix}
\numberwithin{equation}{section} 
\newcounter{cont}[section] 

\newtheorem{thm}[cont]{Theorem}
\newtheorem{prop}[cont]{Proposition}
\newtheorem{lem}[cont]{Lemma}

\theoremstyle{definition}
\newtheorem{defn}[cont]{Definition}
 \theoremstyle{remark}
 \newtheorem{rem}[cont]{Remark}

\newcommand{\R}{\mathbb{R}}
\newcommand{\e}{\varepsilon}

\makeatletter
\@namedef{subjclassname@2020}{\textup{2020} Mathematics Subject Classification}
\makeatother

\begin{document}

\title[Ginzburg--Landau functional with mean curvature operator]{Minimization of a Ginzburg--Landau functional with mean curvature operator in $1$-D}

\author[R. Folino]{Raffaele Folino}

\address[R. Folino]{Departamento de Matem\'aticas y Mec\'anica\\Instituto de 
Investigaciones en Matem\'aticas Aplicadas y en Sistemas\\Universidad Nacional Aut\'onoma de 
M\'exico\\Circuito Escolar s/n, Ciudad Universitaria C.P. 04510 Cd. Mx. (Mexico)}

\email{folino@aries.iimas.unam.mx}

\author[C. Lattanzio]{Corrado Lattanzio$^*$}

\address[C. Lattanzio]{Dipartimento di Ingegneria e Scienze
dell'Informazione e Matematica \\
Universit\`a degli Studi dell'Aquila \\
L'Aquila (Italy)}

\email{corrado.lattanzio@univaq.it}



\thanks{$^*$ Corresponding author}

\maketitle


\begin{abstract} 
The aim of this paper is to investigate the minimization problem related to a Ginzburg--Landau energy functional, 
where in particular a nonlinear diffusion of mean curvature--type is considered, together with a classical double well potential.
A careful analysis of the corresponding Euler--Lagrange equation, equipped with natural boundary conditions and mass constraint, leads to the existence of an unique \emph{Maxwell solution}, namely a monotone increasing solution obtained for small diffusion and close to the so--called \emph{Maxwell point}. 
Then, it is shown that this particular solution (and its reversal) has least energy among all the stationary points satisfying the given mass constraint.
Moreover, as the viscosity parameter tends to zero, it converges to   
 the increasing (decreasing for the reversal) \emph{single interface solution}, namely the constrained minimizer of the corresponding energy without diffusion.
 Connections with Cahn--Hilliard models, obtained in terms of variational derivatives of the total free energy considered here, are also presented.
\end{abstract}


\section{Introduction}\label{sec:intro}
Following the original Van der Walls' theory, the energy of a fluid in a one dimensional vessel is given by the Ginzburg--Landau  functional
\begin{equation}\label{GL-intro}
    E[u]=\int_a^b \left[ \frac{\e^2}{2} u_x^2+ F(u) \right] \, dx,
\end{equation}
where $F(u)$ stands for the free energy per unit volume, and $\e$ is a positive (small) viscosity coefficient. 
The study of the minimization of \eqref{GL-intro} with the constant mass constraint, and the complete description of the corresponding minimizer(s), has been carried out in the seminal paper \cite{CGS}, together with the connection of the latter to the minimizer of the purely stationary case $\e=0$, when interfaces (jump in the density $u$) are allowed without any increase of energy.
We also recall that the same problem in the multidimensional case was addressed in \cite{Modica}.
Following the blueprint of \cite{CGS}, the
main goal of our paper is then to minimize the functional
\begin{equation}\label{eq:energy}
	E[u]=\int_{-1}^1\left[\frac{Q(\e^2 u')}{\e^2}+F(u)\right]\,dx,
\end{equation}
over all $u\in H^1(-1,1)$ satisfying the mass constraint 
\begin{equation}\label{eq:constraint-intro}
	\int_{-1}^1 u(x)\,dx=2r.
\end{equation}
In \eqref{eq:energy}, the function $Q$ is explicitly given by
\begin{equation}\label{eq:Q}
	Q(s)=\sqrt{1+s^2}-1,
\end{equation}
while for the free energy we assume that there exist $0<\overline\alpha<\underline\beta\in\R$ such that the function $F\in C^5(0,\infty)$ satisfies 
\begin{equation}\label{eq:hypF}
	\begin{aligned}
	    & F''>0 \, \mbox{ on } \, (0,\overline\alpha)\cup(\underline\beta,\infty), \qquad F''<0 \, \mbox{ on } \, (\overline\alpha,\underline\beta);\\
        & F'(0)<F'(\underline\beta), \qquad \quad F'(\infty)>F'(\overline{\alpha}).
	\end{aligned}
\end{equation} 
The  above conditions are summarized in Figure \ref{fig:F'}, from which the connection with the  aforementioned 1-$d$ Van der Walls' theory is manifest.
\begin{figure}
\centering
\includegraphics{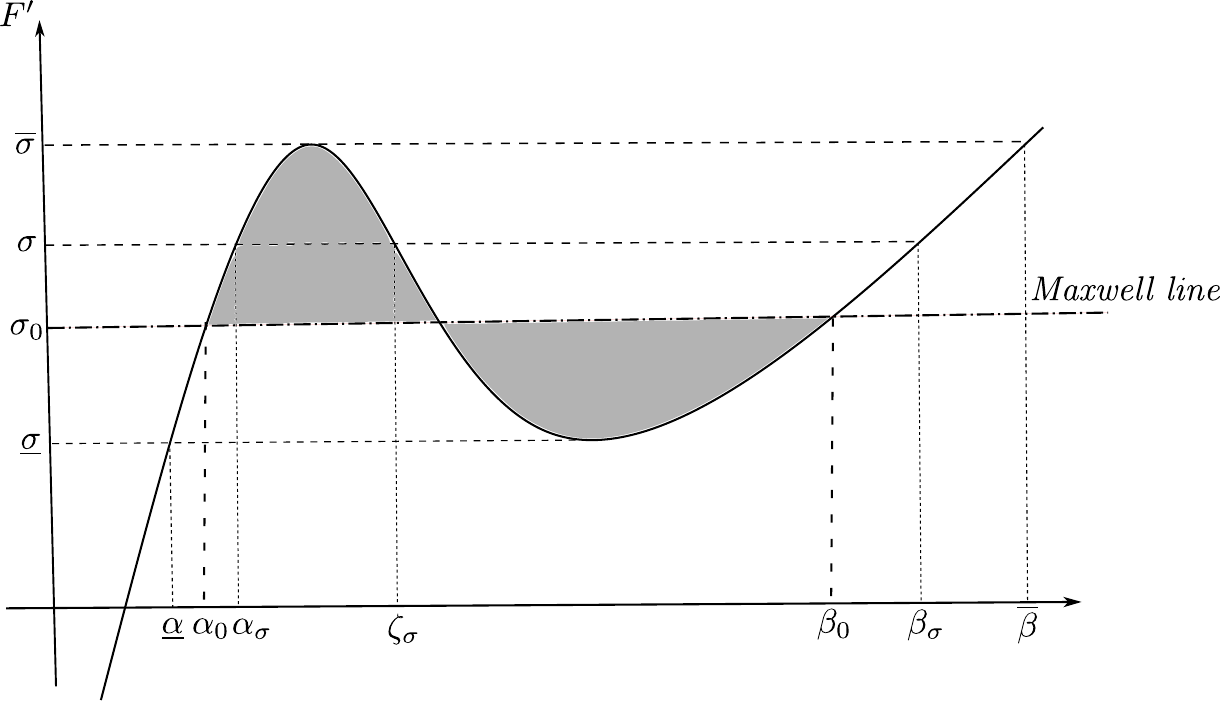}
\caption{Graph of $F'$ and possible choices of $\sigma$.}
\label{fig:F'}
\end{figure}
Moreover, the specific form of the function $Q$ in
the energy functional \eqref{eq:energy} considered in the present paper is motivated by the discussion made  by Rosenau in \cite{Rosenau89,Rosenau90}.  In the theory of phase transitions, in order to include interaction due to high gradients, the author extends the Ginzburg--Landau free--energy functional \eqref{GL-intro} and considers a free--energy functional with $Q$ as in \eqref{eq:Q}, and so, with a linear growth rate with respect to the gradient.
We also mention here the papers  \cite{DPV20,FPS,FS,KurRos,KurRos2}, where the same nonlinear diffusion of mean curvature--type is also considered in different contexts. Finally, 
it is worth observing  that $0\leq Q(s) \leq s^2$ for any $s\in\R$ and therefore \eqref{eq:energy} is well--defined for $u\in H^1(-1,1)$. Nevertheless, we underline that the function $Q$ is convex, but fails to be coercive, and therefore, to the best of our knowledge,  there is no straightforward evidence of existence of minimizers for the functional \eqref{eq:energy}.

\begin{rem}
We point out that the explicit choice \eqref{eq:Q} for the function $Q$, which is the paradigmatic example given in \cite{Rosenau89,Rosenau90}, is made here for simplicity and readability of the paper. However, we claim that the main results contained in this article can be proved for a generic even function $Q\in C^2(\R)$ satisfying
\begin{equation}\label{eq:Q-generic}
    Q(0)=Q'(0)=0, \qquad \lim_{s\to\pm\infty} Q'(s)=\pm1, \qquad Q''(s)>0, \mbox{ for } s\in\R, \qquad 
    \lim_{s\to\pm\infty} Q''(s)=0.
\end{equation}
Notice that, under these assumptions, we readily obtain $0\leq Q(s) \leq C s^2$ for any $s\in\R$, so that the minimization problem under discussion here is still well defined.
Actually, we emphasize that the specific form of the nonlinear function $Q$ is used only in Sections \ref{sec:maxsol}-\ref{sec:proof}, and in particular in the explicit formulas \eqref{eq:P_eps}, \eqref{eq:I_n} and \eqref{eq:splitIn}-\eqref{eq:Rn}.
We sketch the argument to handle the more general case in Remarks \ref{rem:unique} and \ref{rem:generic-Q-4}.
\end{rem}

In the case $\e=0$,  the problem we are dealing with reduces to the 
minimization of
\begin{equation}\label{eq:energy-eps0}
	\int_{-1}^1F(u)\,dx,
\end{equation}
with the constraint \eqref{eq:constraint-intro}, which leads to the study of the auxiliary functional
\begin{equation}\label{eq:energy-aux}
	\int_{-1}^1\left [ F(u) - \sigma u\right]\,dx,
\end{equation}
where $\sigma$ stands for a (constant) Lagrange multiplier.
The function inside the above integral is referred to as 
 \emph{Gibbs function} \cite{CGS}
\begin{equation}\label{eq:Gibbs}
	\Phi_\sigma(z):=F(z)-\sigma z,
\end{equation}
and its   properties, depending on the assumptions \eqref{eq:hypF} on $F$, are crucial in the minimization problems listed above.
\begin{figure}
\centering
\includegraphics{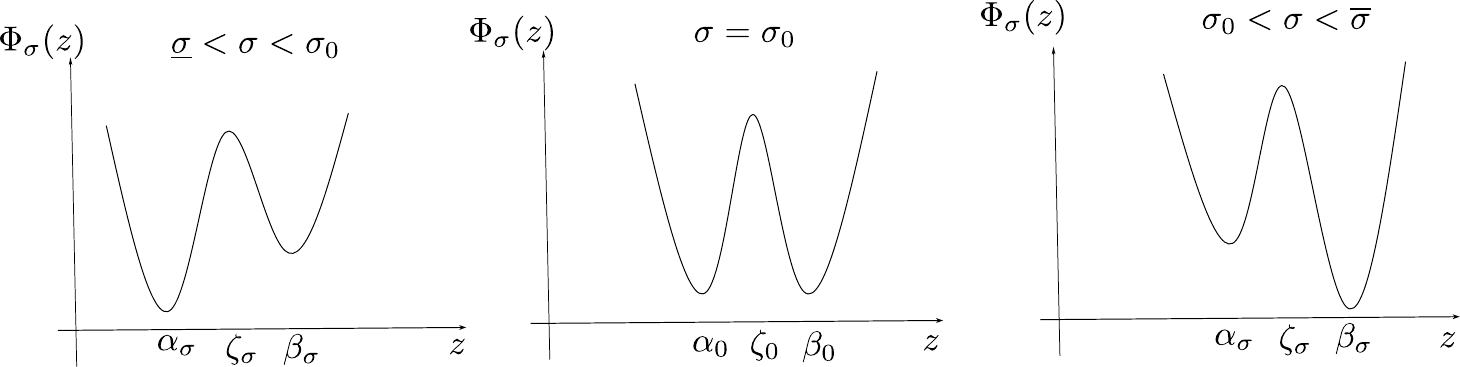}
\caption{Graph of the Gibbs function $\Phi_\sigma$ for different values of $\sigma$.}
\label{fig:CGS-upper}
\end{figure}
We depict in  Figure \ref{fig:CGS-upper} the graph of the Gibbs function for different values of $\sigma\in(\underline\sigma,\overline\sigma)$. The value $\sigma_0$ refers to the so--called \textbf{Maxwell point}, which can be obtained from the celebrated Maxwell construction, also 
known as \emph{equal area rule} \cite{Reichl}; see Figure \ref{fig:F'}.
This construction clearly leads to the following relations:
\begin{equation}\label{eq:ELWE}
\begin{aligned}
     F(\beta_0) - F(\alpha_0) = \sigma_0(\beta_0 - \alpha_0),\\
    \sigma_0 = F'(\alpha_0) = F'(\beta_0),
\end{aligned}
\end{equation}
or, in terms of the Gibbs function, $\Phi_{\sigma_0}(\alpha_0)=\Phi_{\sigma_0}(\beta_0)$.
Hence, the Maxwell point is defined by the following pair of parameters:
\begin{equation}\label{eq:Maxwell-p}
    \Delta_0:=(\sigma_0,b_0), \qquad \qquad \mbox{ with } \quad b_0:=\Phi_{\sigma_0}(\alpha_0)=\Phi_{\sigma_0}(\beta_0).
\end{equation}
Conditions \eqref{eq:ELWE} are related to the minimization problem for \eqref{eq:energy-aux}, for which  the associated Euler--Lagrange equation $F'(u) = \sigma$ should hold at points of continuity of $u$, while the function  $F(u) - \sigma u$   should be continuous  across jumps of $u$ (the \emph{Weierstrass--Erdmann corner condition}).
In addition, as pinpointed  in Figure \ref{fig:CGS-upper},  the Gibbs function for  $\sigma = \sigma_0$ becomes  a double well potential with wells $\alpha_0,\beta_0$ of equal depths.

At this stage, we can define the \emph{single--interface} solutions
\begin{equation}\label{eq:mineps0}
    u_0(x):=\begin{cases}
        \alpha_0, \qquad &-1\leq x\leq -1+\ell_1,\\
        \beta_0, & -1+\ell_1<x\leq1,
    \end{cases}
    \qquad \mbox{ and } \qquad u_0(-x),
\end{equation}
where
\begin{equation}\label{eq:ell}
    \ell_1:=\frac{2(\beta_0-r)}{\beta_0-\alpha_0}, \qquad \qquad 
    \ell_2:=\frac{2(r-\alpha_0)}{\beta_0-\alpha_0}.
\end{equation}
As proven in \cite{CGS}, this particular minimizer of \eqref{eq:energy-eps0} with the constraint \eqref{eq:constraint-intro} is the \emph{physically most relevant}, because, for $\e >0$ sufficiently small,  there exists a unique (modulo reversal) global minimizer of \eqref{GL-intro} with the same constraint; this minimizer is strictly monotone and it converges, as $\e\to 0$, to \eqref{eq:mineps0}. 
The present investigation will lead to the same conclusions concerning stationary points of the generalized energy functional \eqref{eq:energy}, again with the mass constraint \eqref{eq:constraint-intro}. In particular, we shall first characterize all possible smooth minimizers of this problem by solving the corresponding Euler--Lagrange equation with natural boundary conditions. Then, for $\e$ sufficiently small, we shall prove that, among them, there exists a unique (modulo reversal) solution with least energy, which turns out to be monotone and to converge to the same single--interface solutions \eqref{eq:mineps0} as $\e\to0^+$.

\subsection{Plan of the paper}
This paper is organized as follows.
In Section \ref{sec:GF} we connect our minimization problem with a Cahn--Hilliard--type equation, which is derived by using \eqref{eq:energy} instead of the classical Ginzburg--Landau functional \eqref{GL-intro}. Section \ref{sec:maxsol} is devoted to the characterization of all possible smooth minimizers of the functional \eqref{eq:energy} satisfying the mass constraint \eqref{eq:constraint-intro}. Moreover, we shall state in this section the theorem which identifies, for $\e $ small, 
the \textbf{Maxwell solution}, that is a monotone increasing solution of the Euler--Lagrange equation satisfying both the natural boundary conditions and the mass constraint \eqref{eq:constraint-intro}, see Theorem \ref{thm:Maxwell}. The existence of the Maxwell solution is then proved in the subsequent Section \ref{sec:proof}.
Finally, in Section \ref{sec:min} we prove that the Maxwell solution (and its reversal) is the stationary point of \eqref{eq:energy}-\eqref{eq:constraint-intro} with least energy, and it coincides with the unique minimizer of the problem, provided the latter exists.
\subsection{Outline of the main results}
Since our analysis is quite technical, we conclude this Introduction by outlining our main results, the strategy used to prove them, and  the main differences with respect to \cite{CGS}.

It is well known that any smooth stationary point of \eqref{eq:energy} must satisfy the corresponding Euler--Lagrange equation together with natural boundary conditions. In addition, we need to consider the mass constraint \eqref{eq:constraint-intro}, which leads to the problem \eqref{eq:min} in Section \ref{sec:maxsol}.
Then, in order to characterize all stationary points, we first study the Euler--Lagrange equation in the whole real line (hence, without the boundary conditions and the mass constraint). It turns out that there exist periodic solutions for any \textbf{admissible pair} $\Delta=(\sigma,b)$; see Section \ref{sec:maxsol} and specifically Proposition \ref{prop:periodicsol}. Let us underline that the proof of Proposition \ref{prop:periodicsol} is an important preliminary step of our analysis, which is not needed in \cite{CGS}.
Next, in Section \ref{sec:proof} we prove the main result of Section \ref{sec:maxsol}, Theorem \ref{thm:Maxwell}, which establishes the existence of a \emph{special} monotone solution of \eqref{eq:min} (called Maxwell solution) by showing that it is possible to chose the parameters $\sigma, b$ (near the Maxwell point) such that the restriction to the interval $[-1,1]$ of the periodic solution also satisfies the boundary conditions and the mass constraint. Hence, the Maxwell solution is a monotone stationary point of \eqref{eq:energy}-\eqref{eq:constraint-intro}.
The existence and uniqueness of the Maxwell solution is obtained by rewriting the problem
in an appropriate way (see \eqref{eq:I_0,1}) and then by using the Implicit Function Theorem.
To obtain our result, the main extra difficulty with respect to \cite{CGS} is to prove Lemma \ref{Lem:Rn}, which is needed here in view of to the presence  of the nonlinear function $Q$.

We claim that, using the same strategy, it is possible to show the existence of other solutions of \eqref{eq:min},
that is stationary points with an arbitrary number of transitions,
as well as other monotone solutions. However, we will not investigate the existence of these solutions in detail, because we shall see in Section \ref{sec:min} that they can not minimize the energy \eqref{eq:energy} with constraint \eqref{eq:constraint-intro}.
Indeed, we evaluate \eqref{eq:energy} at the Maxwell solution (Proposition \ref{prop:enery-Maxwell}) and show that it minimizes \eqref{eq:energy}-\eqref{eq:constraint-intro} along all  possible stationary points, see Propositions \ref{prop:comp-const}, \ref{prop:less} and \ref{prop:comp-nonmon}.
In particular, we shall prove that the energy \eqref{eq:energy} of the Maxwell solution has the asymptotic expansion $E_r+o(\e)$, when $\e$ is sufficiently small,
where $E_r$ is the minimum of \eqref{eq:energy-eps0}-\eqref{eq:constraint-intro}.
Finally, we show that the Maxwell solution and its reversal converge as $\e\to0^+$ to the single--interface solutions defined in \eqref{eq:mineps0}, see Theorem \ref{thm:eps-conv}.

Summarizing, the main result of this paper is that, if $r\in(\alpha_0,\beta_0)$, then
\begin{itemize}
    \item[(i)] when $\e>0$ is small enough, there exists a smooth strictly increasing stationary point $u_\e:=u_\e(x)$ of \eqref{eq:energy}-\eqref{eq:constraint-intro} and $u_\e$ and its reversal $u^R_\e:=u^R_\e(x)=u_\e(-x)$ are the stationary points with least energy;
    \item[(ii)] $E[u_\e]=E[u_\e^R]$ converges as $\e\to0^+$ to the minimum of \eqref{eq:energy-eps0}, that is the minimum energy in the case without diffusion;
    \item[(iii)] $u_\e$ and $u_\e^R$ converge as $\e\to0^+$ to the single--interface solutions defined in \eqref{eq:mineps0}, namely the two--phase solutions with least energy in the case without diffusion.
\end{itemize}

\section{Connections with Cahn--Hilliard models}\label{sec:GF}
As a possible motivation for our studies, in this section we derive the evolutionary equation related to the minimization problem under investigation in the paper. 
In particular, 
we show that a Cahn--Hilliard type equation with nonlinear diffusion can be derived from the energy functional \eqref{eq:energy}.

First, let us recall that the celebrated Cahn--Hilliard equation, which in the one--dimensional case reads as
\begin{equation}\label{eq:Ca-Hi}
	u_t=(-\e^2 u_{xx}+F'(u))_{xx}, \qquad \quad 
	x\in(a,b), \; t>0,
\end{equation} 
where $\e>0$ and $F:\R\to\R$ is a double well potential with wells of equal depth,
has been originally proposed in \cite{Cahn,CH} to model phase separation in a binary system at a fixed temperature, with constant
total density and where $u$ stands for the concentration of one of the two components.
Generally, \eqref{eq:Ca-Hi} is considered with homogeneous Neumann boundary conditions
\begin{equation}\label{eq:Neu}
	u_x(a,t)=u_x(b,t)=u_{xxx}(a,t)=u_{xxx}(b,t)=0, \qquad \qquad t>0,
\end{equation} 
which are physically relevant since they  
guarantee that the total mass is conserved.

Moreover, it is well-known \cite{Fife} that equation \eqref{eq:Ca-Hi} is the gradient 
flow in the zero-mean subspace of the dual of $H^1(a,b)$ of the Ginzburg--Landau functional \eqref{GL-intro}
and that the only {\it stable} stationary solutions to \eqref{eq:Ca-Hi}-\eqref{eq:Neu} are minimizers of  the energy \eqref{GL-intro} \cite{Zheng}. 
Therefore, thanks to the aforementioned work \cite{CGS}, we can state that solutions to \eqref{eq:Ca-Hi}-\eqref{eq:Neu} converge, as $t \to \infty$, to a limit which has at most a single transition inside the interval $[a,b]$.
It is also worth mentioning that the convergence to the monotone steady states is incredibly slow and if the initial profile has an $N$-transition layer structure, 
oscillating between the two minimal points of $F$, then the solution maintains the unstable structure for a time $T_\e=\mathcal{O}\left(\exp(c/\e)\right)$, as $\e\to0^+$, see \cite{AlikBateFusc91,Bates-Xun1,Bates-Xun2}.
This phenomenon is known in literature as \emph{metastable dynamics}, and it is studied in various articles and with different techniques for reaction--diffusion models, among others see \cite{FPS,FS} and references therein. The analysis of this property can be carried out also for equation \eqref{eq:Ca-Hi} with nonlinear diffusions, as in the case under investigations here, but it is not in the main aims of the present paper and it is left for future investigations; the interested reader can refer to  \cite{FLS2022} for the case of the $p$--Laplace operator.

Equation \eqref{eq:Ca-Hi} can be also seen as the simplest 1-$d$ case of a very general Cahn--Hilliard model introduced by Gurtin in \cite{Gurtin}, which reads as
\begin{equation}\label{eq:Gurtin-model}
	u_t=\mbox{div}\left\{D\nabla\left[-\mbox{div}[\partial_v {\Psi}(u,\nabla u)]+\partial_u{\Psi}(u,\nabla u)-\gamma\right]\right\}+m,
\end{equation}
where $D$ is a non constant \emph{mobility} (which may depends on $u$ and its derivatives),  $\Psi$ is the so-called \emph{free energy}, $\gamma$ is an \emph{external microforce} and $m$ is an \emph{external mass supply}, for further details see \cite{Gurtin}.
In particular, the standard Cahn--Hilliard equation \eqref{eq:Ca-Hi} corresponds to the one-dimensional version of \eqref{eq:Gurtin-model}, with the choices $D\equiv1$, $\gamma\equiv m\equiv0$ and the free energy
\begin{equation}\label{eq:free_en_st}
    \Psi(u,v):=\frac{\e^2}{2}|v|^2+F(u).
\end{equation}
If we consider a concentration--dependent mobility (cfr. \cite{Ca-El-NC} and references therein) $D(u) : \R\to \R^+$, $\gamma=m=0$ as in the standard case and the free energy
\begin{equation}\label{eq:free-energy}
	\Psi(u,v):=\e^{-2}Q (\e^2|v|)+F(u)=\e^{-2}\sqrt{1+\e^4|v|^2}-\e^{-2}+F(u),
\end{equation}
we end up with the following Cahn--Hilliard model with mean curvature type diffusion
\begin{equation}\label{eq:CH-model-Q}
	u_t=\mbox{div}\left\{D(u)\nabla\left[-\mbox{div}\left(Q'(\e^2|\nabla u|)\frac{\nabla u}{|\nabla u|}\right)+F'(u)\right]\right\}.
\end{equation}

For the sake of completeness, we recall here how one can derive the model \eqref{eq:CH-model-Q} from the (multi--$d$ version of the) functional \eqref{eq:energy}.
Let us consider the \emph{total (integrated) free energy} \cite{Gurtin}
\begin{equation*}
    \mathcal{E}[u]:=\int_\Omega\Psi(u,\nabla u)\,dx.
\end{equation*}
The formal variation with respect to fields $\phi$ that vanish on the boundary $\partial\Omega$ is given by
\begin{align*}
    \delta\mathcal{E}&=\left\langle\frac{\delta\mathcal{E}}{\delta u},\phi\right\rangle=\frac{d}{ds}\mathcal{E}[u+s\phi]\Bigg|_{s=0}=
    \frac{d}{ds}\int_\Omega\Psi(u+s\phi,\nabla u+s\nabla\phi)\,dx\Bigg|_{s=0}\\
    &=\int_\Omega\left[\Psi_u(u,\nabla u)\phi+\Psi_v(u,\nabla u)\cdot\nabla\phi\right]\,dx\\
    &=\int_\Omega\left[\Psi_u(u,\nabla u)-\mbox{div}\Psi_v(u,\nabla u)\right]\phi\,dx,
\end{align*}
where $\langle\cdot,\cdot\rangle$ denotes the $H^{-1}$, $H_0^1$ pairing. As a consequence, the variational derivative is given by
\begin{equation}\label{eq:var_der}
    \frac{\delta\mathcal{E}}{\delta u}= \Psi_u(u,\nabla u)-\mbox{div}\Psi_v(u,\nabla u).
\end{equation}
The standard Cahn--Hilliard equation is derived from
\begin{equation}\label{eq:CH-var}
    u_t=\mbox{div}\left[D(u)\nabla \frac{\delta \mathcal{E}}{\delta u}\right],
\end{equation}
where $D\equiv1$ and the free energy is given by \eqref{eq:free_en_st};
indeed, in this case one has
\begin{equation*}
    \Psi_u(u,\nabla u)=F'(u), \qquad \qquad \Psi_v(u,\nabla u)=\e^2\nabla u.
\end{equation*}
On the other hand, in the case of free energy defined as in \eqref{eq:free-energy}, one obtains the total free energy
\begin{equation}\label{eq:energy-multid}
    \mathcal{E}[u]:=\int_\Omega\left[\frac{Q(\e^2|\nabla u|)}{\e^{2}}+F(u)\right]\,dx,
\end{equation}
that is, the multi-dimensional version of \eqref{eq:energy}, and since
\begin{equation*}
    \Psi_u(u,\nabla u)=F'(u), \qquad \qquad \Psi_v(u,\nabla u)=Q'(\e^2\nabla u)\frac{\nabla u}{|\nabla u|},
\end{equation*}
substituting in \eqref{eq:var_der}-\eqref{eq:CH-var}, one deduces the model
\begin{equation}\label{eq:CH-model}
	u_t=\mbox{div}\left\{D(u)\nabla\left[-\mbox{div}\left(\frac{\e^2\nabla u}{\sqrt{1+\e^4|\nabla u|^2}}\right)+F'(u)\right]\right\},
\end{equation}
that is \eqref{eq:CH-model-Q} with the explicit formula $Q'(s)=\displaystyle\frac{s}{\sqrt{1+s^2}}$.

As an alternative viewpoint for the derivation of \eqref{eq:CH-model}, we recall the continuity equation for the concentration $u$
\begin{equation}\label{eq:cont}
		u_t+\mbox{div}J=0,
\end{equation}
where $J$ is its flux. 
The standard Cahn--Hilliard equation follows from \eqref{eq:cont} and the  constitutive equation
\begin{equation*}
    J=-\nabla\mu,
\end{equation*}
which relates the flux $J$ to the chemical potential \cite{Gurtin}
\begin{equation}\label{eq:flux}
	\mu=\frac{\delta\mathcal{E}}{\delta u}=-\e^2\Delta u+F'(u).
\end{equation}
On the other hand, by considering (again) the variational derivative \eqref{eq:var_der} with free energy \eqref{eq:free-energy} and a concentration dependent mobility, we obtain the flux
\begin{equation}\label{eq:flux2}
	J=-D(u)\nabla\mu, \qquad \mbox{ with } \qquad \mu=\frac{\delta\mathcal{E}}{\delta u}=\left\{-\mbox{div}\left(\frac{\e^2\nabla u}{\sqrt{1+\e^4|\nabla u|^2}}\right)+F'(u)\right\}.
\end{equation}
By substituting \eqref{eq:flux2} in  the continuity equation \eqref{eq:cont}, we end up with \eqref{eq:CH-model}.
We emphasize that there are two novelties in \eqref{eq:flux2}: the concentration--dependent  mobility and the mean curvature operator replace the constant mobility and the Laplacian, which are a peculiarity of the classical choice \eqref{eq:flux}.

Let us now show that the physically relevant (no-flux) boundary conditions for \eqref{eq:CH-model} are
\begin{equation}\label{eq:noflux}
    \nabla u\cdot\nu=0, \qquad  J\cdot\nu=0, \qquad \qquad \mbox{ on } \quad \partial\Omega,
\end{equation}
where $\nu$ is the unit normal vector to $\Omega$.
Indeed, the boundary conditions \eqref{eq:noflux} guarantee that the energy \eqref{eq:energy-multid} is a non-increasing function of time along the solutions to \eqref{eq:CH-model} and that these solutions preserve the mass.
Precisely, differentiating \eqref{eq:energy-multid} with respect to time and using the first condition in \eqref{eq:noflux}, we get
\begin{align*}
    \frac{d}{dt}\mathcal{E}[u](t)&=\int_\Omega\left[Q'(\e^2|\nabla u|)\frac{\nabla u}{|\nabla u|}\cdot \nabla u_t+F'(u)u_t\right]\,dx\\
    &=\int_\Omega\left\{-\mbox{div}\left[\frac{\e^2\nabla u}{\sqrt{1+\e^4|\nabla u|^2}}\right]+F'(u)\right\}u_t\,dx.
\end{align*}
Moreover, \eqref{eq:cont}-\eqref{eq:flux2} and the second condition in \eqref{eq:noflux} imply
\begin{align*}
    \frac{d}{dt}\mathcal{E}[u](t)&=-\int_\Omega\mu\,\mbox{div}J\,dx =-\int_{\partial\Omega}\mu\,J\cdot\nu\,d\sigma+\int_\Omega\nabla\mu\cdot J\,dx\\
    &=-\int_\Omega D(u)|\nabla\mu|^2\,dx\leq 0,     \qquad \qquad \mbox{ for }\, t>0,
\end{align*}
where we used the positivity of the function $D$.
Finally, by using \eqref{eq:cont} and \eqref{eq:noflux} we infer
\begin{equation*}
    \frac{d}{dt}\int_\Omega u\,dx=\int_\Omega u_t\,dx=-\int_\Omega\mbox{div}J\,dx=-\int_{\partial\Omega}J\cdot\nu\,d\sigma=0,
\end{equation*}
for any $t>0$.

In the one-dimensional case, we expect that the aforementioned results of \cite{Zheng}, valid for \eqref{eq:Ca-Hi}-\eqref{eq:Neu}, can be extended to the boundary value problem \eqref{eq:CH-model}-\eqref{eq:noflux}, that is its solutions converge as $t\to+\infty$ to 
the minimizers of \eqref{eq:energy}-\eqref{eq:constraint-intro}. 
In the following sections, we rigorously prove that, if $\e>0$ is sufficiently small, the stationary points of \eqref{eq:energy}-\eqref{eq:constraint-intro} with least energy are monotone and, as a consequence, we expect that any solution of \eqref{eq:CH-model}-\eqref{eq:noflux} will eventually have at most one transition.
However, inspired by the classical results \cite{AlikBateFusc91,Bates-Xun1,Bates-Xun2} and using the same strategy of \cite{FLS2022,FPS,FS}, 
we are confident that, in the one-dimensional-case, there exist \emph{metastable solutions} of \eqref{eq:CH-model}-\eqref{eq:noflux}, which maintain an unstable structure with an arbitrary number of transitions for an extremely long time.
Up to our knowledge, these results concerning the asymptotic behavior of the solutions are an open problem, which needs further investigation.

\section{Maxwell solution}\label{sec:maxsol}
As it was already mentioned, the main goal of this paper is to minimize the functional
\begin{equation*}
	E[u]=\int_{-1}^1\left[\frac{Q(\e^2 u')}{\e^2}+F(u)\right]\,dx,
\end{equation*}
over all $u\in H^1(-1,1)$, $u>0$, such that
\begin{equation}\label{eq:constraint}
	\int_{-1}^1 u(x)\,dx=2r.
\end{equation}
A classical result of calculus of variations asserts that all smooth stationary points $u=u(x)$ satisfy the Euler--Lagrange equation
$$\frac{d}{dx}\left[Q'(\e^2u')\right]-F'(u)=0,\qquad \qquad \mbox{ in } (-1,1),$$
with homogeneous Neumann boundary conditions $u'(\pm1)=0$.
Expanding the derivative, imposing the constraint \eqref{eq:constraint} and introducing the function
$$z(x)=u(\e x), \qquad \mbox{ or equivalently } \qquad u(x)=z(x/\e),$$
we obtain the problem
\begin{equation}\label{eq:min}
	\begin{cases}
		Q''(\e z')z''=F'(z)-\sigma, \qquad  \mbox{ in } (-\e^{-1},\e^{-1}),\\
		z'(\pm\e^{-1})=0,\\		
		\displaystyle\int_{-\e^{-1}}^{\e^{-1}}z(x)\,dx=2r\e^{-1},
	\end{cases}
\end{equation}
where the constant $\sigma$ is a Lagrange multiplier.
In order to determine all solutions of \eqref{eq:min}, we multiply by $z'$ the ordinary differential equation in \eqref{eq:min} and study the equation 
\begin{equation}\label{eq:min-firstode}
	P_\e(z')=\Phi_\sigma(z)-b,
\end{equation}
in the whole real line,
where $b$ is a constant and we introduced the function 
\begin{equation}\label{eq:P_eps}
	P_\e(s):=\int_0^s Q''(\e u)u\,du=\e^{-1}Q'(\e s)s-\e^{-2}Q(\e s)=\e^{-2}\left(1-\frac{1}{\sqrt{1+\e^2s^2}}\right),
\end{equation}
and the Gibbs function $\Phi_\sigma$, depicted in Figure \ref{fig:CGS-upper}, is defined in \eqref{eq:Gibbs}.
As pointed out in \cite{CGS} (see Propositions 2.1 and 4.1), we briefly recall here that for each $\sigma\in(\underline\sigma,\overline\sigma)$:
\begin{enumerate}
\item $\Phi_\sigma$ has exactly three critical points, namely $\Phi'_\sigma(\alpha_\sigma)=\Phi'_\sigma(\zeta_\sigma)=\Phi'_\sigma(\beta_\sigma)=0$, with $0<\underline\alpha<\alpha_\sigma<\zeta_\sigma<\beta_\sigma<\overline\beta$; 
\item $\alpha_\sigma$, $\zeta_\sigma$, and $\beta_\sigma$, as functions of $\sigma$, are continuous on $[\underline\sigma,\overline\sigma]$, $C^4$ on $(\underline\sigma,\overline\sigma)$, with $\alpha_\sigma$, $\beta_\sigma$ strictly increasing, and $\zeta_\sigma$ strictly decreasing;
\item $\Phi_\sigma$ is strictly decreasing on $(0,\alpha_\sigma)\cup(\zeta_\sigma,\beta_\sigma)$ and strictly increasing on $(\alpha_\sigma,\zeta_\sigma)\cup(\beta_\sigma,+\infty)$;
\item $\Phi_\sigma(\beta_\sigma)>\Phi_\sigma(\alpha_\sigma)$ for $\sigma<\sigma_0$, $\Phi_\sigma(\alpha_\sigma)>\Phi_\sigma(\beta_\sigma)$ for $\sigma>\sigma_0$, and
$\Phi_{\sigma_0}(\alpha_0)=\Phi_{\sigma_0}(\beta_0)$, where $\alpha_0:= \alpha_{\sigma_{0}}$, $\beta_0:= \beta_{\sigma_{0}}$, $\zeta_0:= \zeta_{\sigma_{0}}$;
\end{enumerate}
see also Figure \ref{fig:F'}, where in particular the behavior of the critical points  $\alpha_\sigma$, $\zeta_\sigma$, and $\beta_\sigma$, summarized in points (1) and (2), is manifest.
These properties are instrumental in the study of the behavior of the solutions to the ODE
\begin{equation}\label{eq:ODE-CGS}
    (z')^2=\Phi_\sigma(z)-b,
\end{equation}
that is the ODE satisfied by the minimizers in the case of the classical Ginzburg--Landau functional \eqref{GL-intro}
(the ODE can be obtained by \eqref{eq:min-firstode}-\eqref{eq:P_eps} in the case $Q(s)=s^2$). 
\begin{figure}
\centering
\includegraphics{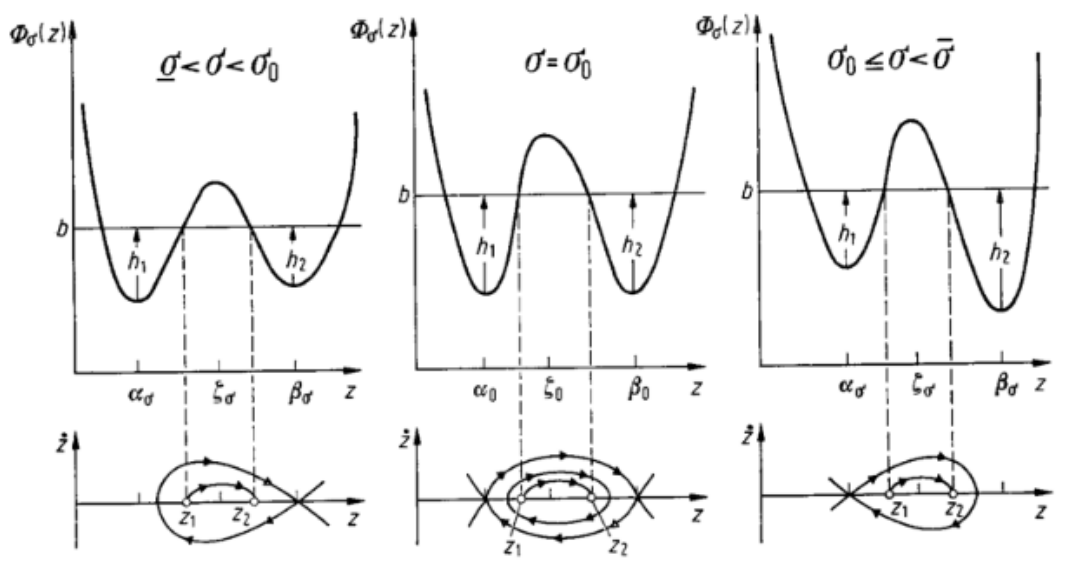}
\caption{Gibbs function and phase portraits for different values of $\sigma\in(\underline\sigma,\overline\sigma)$.
This figure was produced in \cite{CGS}.}
\label{fig:CGS}
\end{figure}
In particular, it is possible to choose the constants $\sigma$, $b$  such that there exist periodic solutions of \eqref{eq:ODE-CGS}, see the phase portraits in Figure \ref{fig:CGS}.
To be more precise, $\sigma$ has to belong to the interval $(\underline\sigma,\overline\sigma)$ and $b$ must satisfy
\begin{equation*}
	\begin{aligned}
	\Phi_\sigma(\beta_\sigma)<b<\Phi_\sigma(\zeta_\sigma), \qquad \mbox{ for } \quad \sigma\in(\underline\sigma,\sigma_0],\\
	\Phi_\sigma(\alpha_\sigma)<b<\Phi_\sigma(\zeta_\sigma), \qquad \mbox{ for } \quad \sigma\in[\sigma_0,\overline\sigma),
\end{aligned}
\end{equation*}
or, equivalently
\begin{equation}\label{eq:b}
    \max \{\Phi_\sigma(\alpha_\sigma), \Phi_\sigma(\beta_\sigma)\}<b<\Phi_\sigma(\zeta_\sigma),
\end{equation}
for any $\sigma\in[\underline\sigma,\overline\sigma]$, in view of point (4) above.
We refer to such a pair $\Delta=(\sigma,b)$ as \textbf{admissible}.
It is worth observing here that the pair $\Delta_0$, corresponding to the Maxwell point defined in \eqref{eq:Maxwell-p}, is not an admissible pair. Indeed, notice that $\Phi_{\sigma_0}$ is a double well potential with wells $\alpha_0,\beta_0$ of equal depths, and, as a consequence, the choice $b_0=\Phi_{\sigma_0}(\alpha_0)=\Phi_{\sigma_0}(\beta_0)$ corresponds to the heteroclinic orbit connecting the two wells $\alpha_0$ and $\beta_0$, see Figure \ref{fig:CGS}.
Hence, the corresponding solution never satisfies the Neumann boundary condition in \eqref{eq:min}.

On the other hand, to any admissible pair $\Delta$ corresponds a periodic solution to \eqref{eq:ODE-CGS};
our goal is to extend the latter result to equation \eqref{eq:min-firstode}.
To this aim, we rewrite this equation as 
\begin{equation}\label{eq:f_delta}
		P_\e(z')= \Phi_\sigma(z)-b =: f_\Delta(z).
\end{equation}
The crucial properties of the function $P_\e$ defined in \eqref{eq:P_eps}, in particular its invertibility, needed to investigate equation \eqref{eq:min-firstode} (or \eqref{eq:f_delta}) are consequences of the bounds for $f_\Delta(z)$ collected in the next lemma.
\begin{lem}\label{lem:fdelta}
For any $z\in[\underline\alpha, \overline\beta]$ and for any $\Delta$ admissible (hence, satisfying \eqref{eq:b}), the function $f_\Delta(z)$ defined in \eqref{eq:f_delta} satisfies
\begin{equation}\label{eq:fdelta_bounds}
  0\leq f_\Delta(z)\leq  \Phi_{\sigma_0}(\zeta_{0})- \Phi_{\sigma_0}(\beta_0).
\end{equation}
\end{lem}
\begin{proof}
We start by observing that condition \eqref{eq:b} in particular guarantees $f_\Delta(z)\geq0$ for any admissible pair $\Delta$ and the first inequality in \eqref{eq:fdelta_bounds} is trivially verified.

Moreover, to estimate the  maximum of $f_\Delta(z)$, we first notice  that the definition \eqref{eq:Gibbs} of $\Phi_\sigma$  and properties (1) and (2) imply for any $\sigma\in(\underline\sigma,\overline\sigma)$
\begin{equation}\label{eq:punticritici}
    \frac{d}{d\sigma}\Phi_\sigma(\xi_\sigma)= (F'(\xi_\sigma) - \sigma) \frac{d}{d\sigma}\xi_\sigma -\xi_\sigma = -\xi_\sigma,
\end{equation}
for any critical point $\xi_\sigma=\alpha_\sigma, \zeta_\sigma, \beta_\sigma$. Then, from \eqref{eq:b} we readily have
\begin{equation*}
    f_\Delta(z)\leq\Phi_\sigma(\zeta_\sigma)-\max\left\{ \Phi_\sigma(\alpha_\sigma),\Phi_\sigma(\beta_\sigma)\right\}=:g(\sigma),
\end{equation*}
for any admissible pair $\Delta$ and any $z\in[\underline\alpha, \overline\beta]$.
To be more precise, we can rewrite
\begin{equation*}
    g(\sigma)=\begin{cases}
    \Phi_\sigma(\zeta_\sigma)- \Phi_\sigma(\beta_\sigma), \qquad \qquad & \sigma\in[\underline\sigma,\sigma_0],\\
    \Phi_\sigma(\zeta_\sigma)-\Phi_\sigma(\alpha_\sigma), & \sigma\in[\sigma_0,\overline\sigma].
    \end{cases}
\end{equation*}
Then, from \eqref{eq:punticritici} it follows that
\begin{equation*}
    g'(\sigma)=\begin{cases}
    -\zeta_\sigma+\beta_\sigma, \qquad \qquad & \sigma\in(\underline\sigma,\sigma_0),\\
    -\zeta_\sigma+\alpha_\sigma, &\sigma\in(\sigma_0,\overline\sigma),
    \end{cases}
\end{equation*}
and, using again (1), we conclude that $g'(\sigma)>0$ for $\sigma\in(\underline\sigma,\sigma_0)$ and $g'(\sigma)<0$ for $\sigma\in(\sigma_0,\overline\sigma)$.
Hence, since from (4) $g(\underline\sigma)=g(\overline\sigma)=0$ and $g(\sigma_0)=\Phi_{\sigma_0}(\zeta_{0})- \Phi_{\sigma_0}(\beta_0)=\Phi_{\sigma_0}(\zeta_{0})- \Phi_{\sigma_0}(\alpha_0)$,
the function $f_\Delta(z)$ defined in \eqref{eq:f_delta} satisfies \eqref{eq:fdelta_bounds} 
for any $z\in[\underline\alpha, \overline\beta]$ and for any $\Delta$ admissible.
\end{proof}

The upper bound proved in Lemma \ref{lem:fdelta}, depending solely on the potential $F$,  justifies the following definition, which will be used in the sequel:
\begin{equation}\label{eq:eps-small-final}
  \overline F : =  \big [\Phi_{\sigma_0}(\zeta_{0})- \Phi_{\sigma_0}(\beta_0)\big]^{-1/2}.
\end{equation}
\begin{prop}\label{prop:periodicsol}
Assume $F$ satisfies \eqref{eq:hypF} and let $P_\e$ be defined by \eqref{eq:P_eps}.
Moreover, assume $ \e \in (0,\overline{F})$, where $\overline{F}$ is defined in \eqref{eq:eps-small-final}.
Then, for any  $\Delta=(\sigma,b)$  admissible,  equation \eqref{eq:f_delta} admits periodic solutions. 
\end{prop}
\begin{proof}
We first notice that $P_\e$, defined in \eqref{eq:P_eps}, is an even function satisfying 
\begin{equation*}
	P_\e(0)=0, \qquad P_\e'(s)s>0, \quad s\neq0, \qquad \lim_{s\to\pm\infty}P_\e(s)=\e^{-2}.
\end{equation*}
Therefore, $P_\e(s) = \xi$ is invertible for $s\gtrless 0$ with inverses $H_\e^\pm(\xi)$ given  by
\begin{equation}\label{eq:defP^-1}
    H^{\pm}_\e(\xi):= \pm\frac{\sqrt{\xi(2-\e^2\xi)}}{1-\e^2\xi},
\end{equation}
provided $0<\xi<\e^{-2}$. 
Hence, thanks to the bounds contained in Lemma \ref{lem:fdelta}, for any $\Delta$ admissible, for any $z$ under consideration, and choosing  $0<\e<\overline F$, we can invert the function $P_\e(f_\Delta(z))$ in \eqref{eq:f_delta}. 
More precisely, this equation is equivalent to
\begin{equation}\label{eq:z'=+}
    z'=H^+_\e\left(\Phi_\sigma(z)-b)\right),
\end{equation}
respectively
\begin{equation}\label{eq:z'=-}
    z'=H^-_\e\left(\Phi_\sigma(z)-b)\right),
\end{equation}
for monotone increasing, respectively decreasing, solutions.
The solutions of \eqref{eq:z'=+}-\eqref{eq:z'=-} are implicitly defined by
\begin{equation}\label{eq:implicit-z}
    \int_{z(0)}^{z(x)}\frac{ds}{H^\pm_\e\left(\Phi_\sigma(s)-b\right)}=x.
\end{equation}
At this point, let us fix $z_1:=z_1(\Delta)$ and $z_2:=z_2(\Delta)$ satisfying 
\begin{equation*}
    z_1<z_2, \qquad \qquad \Phi_\sigma(z_1)=\Phi_\sigma(z_2)=b.
\end{equation*}
We claim that, if $\Delta=(\sigma,b)$ is admissible, then 
\begin{equation}\label{eq:finite}
    \int_{z_1}^{z_2}\frac{ds}{H_\e^{+}\left(\Phi_\sigma(s)-b\right)}<+\infty
\end{equation}
ensures that there exist periodic solutions oscillating between $z_1$ and $z_2$. 
For definiteness, assume $z(0)=z_1$;
if \eqref{eq:finite} holds true, then the function $z$ implicitly defined by \eqref{eq:implicit-z} is monotone increasing, $z(T(\Delta))=z_2$ and satisfies \eqref{eq:f_delta} in the interval $[0,T(\Delta)]$, where
\begin{equation}\label{eq:tdelta}
    T(\Delta):= \int_{z_1}^{z_2}\frac{ds}{H_\e^{+}\left(\Phi_\sigma(s)-b\right)}<+\infty.
\end{equation}
Similarly, taking the inverse $H_\e^- = - H_\e^+ $, we can extend the solution in the interval $[T(\Delta),2T(\Delta)]$.
In particular, $z$ is monotone decreasing in $[T(\Delta),2T(\Delta)]$,
\begin{equation*}
    T(\Delta)=\int_{z_2}^{z_1}\frac{ds}{H_\e^{-}\left(\Phi_\sigma(s)-b\right)}<+\infty,
\end{equation*}
and $z(2T(\Delta))=z_1$.
Then, we can extend the solution to the whole real line by $2T(\Delta)$--periodicity.

We are left with the proof of \eqref{eq:finite}, namely, of the boundedness of the improper integral defining $T(\Delta)$. 
Since $H^+_\e(s)=\sqrt{2s}+o(s)$, as $s\to0^+$, using the definition \eqref{eq:f_delta}, we deduce
\begin{equation*}
     \int_{z_1}^{z_2}\frac{ds}{H_\e^+\left(\Phi_\sigma(s)-b\right)}=
     \int_{z_1}^{z_2}\frac{ds}{H_\e^+\left(f_\Delta(s)\right)}\sim\int_{z_1}^{z_2}\frac{ds}{\sqrt{f_\Delta(s)}}.
\end{equation*}
However, $f_\Delta(z_1)=f_\Delta(z_2)=0$ and we can expand $f_\Delta$ close to $z_i$, $i=1,2$ to obtain
\begin{align*}
    f_\Delta(s)&=(F'(z_1)-\sigma)(s-z_1)+o(|s-z_1|), \qquad s\sim z_1,\\
    f_\Delta(s)&=(F'(z_2)-\sigma)(s-z_2)+o(|s-z_2|), \qquad s\sim z_2.
\end{align*}
Finally, $F'(z_i)-\sigma\neq0$ for $i=1,2$ because $z_i$ are not critical points of the Gibbs function $\Phi_\sigma$, and in particular $\Phi'_\sigma(z_1)>0$ and $\Phi'_\sigma(z_2)<0$.
As a consequence, $T(\Delta)$ is bounded, being
\begin{equation*}
    \int_{z_1}^{z_2}\frac{ds}{\sqrt{f_\Delta(s)}}< + \infty,
\end{equation*}
and the proof is complete.
\end{proof}
\begin{rem}\label{rem:unique}
Proposition \ref{prop:periodicsol} guarantees the existence of periodic solutions in the whole real line, which is instrumental in constructing particular solutions of \eqref{eq:min}, as we shall see below. In addition, we emphasize that, being $P_\e (s) = 0$ if and only if $s=0$, the phase portrait of equation \eqref{eq:f_delta} is (qualitatively) the same of the one depicted in Figure \ref{fig:CGS}, and therefore all possible solutions of \eqref{eq:min} are constants or restriction of periodic solutions in the real line, with half period $T(\Delta)$ defined in \eqref{eq:tdelta} appropriately chosen to fulfill the boundary conditions, that is 
$n T(\Delta)=2\e^{-1}$, $n\in\mathbb{N}$.  In particular, if  $T(\Delta) = 2\e^{-1}$, then we obtain a monotone solution of \eqref{eq:min}, while for $n\geq 2$ we obtain all other possible (non constant and non monotone) solutions.

Finally, we emphasize that the crucial property needed   in the proof of Proposition \ref{prop:periodicsol} is the bound \eqref{eq:finite}. In the present case,  we prove it taking advantage of the  given formula \eqref{eq:Q} for $Q$ to compute $P_\e$ and its inverses $H^{\pm}_\e$ given by \eqref{eq:defP^-1} explicitly.
Hence, to extend the existence result to the case of a generic $Q$ satisfying \eqref{eq:Q-generic}, it is sufficient to prove that $P_\e$ has inverses which satisfy $H^{\pm}_\e(\xi)\approx\e^{-1}\sqrt{\xi}$, when $\xi\sim0$, which clearly implies \eqref{eq:finite}.
The aforementioned behavior close to zero of $H^{\pm}_\e(\xi)$ can be proved proceeding as in \cite[Lemma 2.1]{DFS}.
\end{rem}

The next goal is to prove that we can choose an admissible $\Delta$ such that $T(\Delta) = 2\e^{-1}$, namely the existence of monotone solutions to \eqref{eq:min}.
For definiteness, we focus our attentions to (non--constant) increasing solutions of \eqref{eq:min}, referred to as \textbf{simple solutions}, following  \cite{CGS}.
Specifically, in order to select such $\Delta$, we need to find conditions on $\sigma$ and $b$ such that the restriction in $(-\e^{-1},\e^{-1})$ of the corresponding periodic solution satisfies the boundary conditions and the mass constraint in \eqref{eq:min}.
In particular, the boundary conditions imply $\Phi_\sigma(z(\pm\e^{-1}))=b$ or, equivalently,
\begin{equation*}
	\Phi_\sigma(z_1)=\Phi_\sigma(z_2)=b,
\end{equation*}
where 
\begin{equation}\label{eq:boundary-z}
	z_1=z(-\e^{-1}), \qquad \mbox{ and } \qquad z_2=z(\e^{-1}).
\end{equation}
Both conditions can be recast as 
\begin{equation}\label{eq:I_0,1}
\begin{aligned}
    I_0(\Delta;\e) &= \int_{-\e^{-1}}^{\e^{-1}} \,dx = 2\e^{-1},\\
    I_1(\Delta;\e)&=\int_{-\e^{-1}}^{\e^{-1}} z(x)\,dx=2r\e^{-1},
\end{aligned}
\end{equation}
where we use the notation
\begin{equation*}
    I_n(\Delta;\e) : =	\int_{-\e^{-1}}^{\e^{-1}}z(x)^n\,dx.
\end{equation*}
In other words, the existence of a simple solution to \eqref{eq:min} is equivalent to find $\Delta$ admissible which solves system \eqref{eq:I_0,1}.
To this aim, since a simple solution satisfies $z'(x)>0$ for any $x\in(-\e^{-1},\e^{-1})$, we take advantage of \eqref{eq:boundary-z} and \eqref{eq:f_delta} to convert the above integral
into one of the form
\begin{equation}\label{eq:I_n}
	I_n(\Delta;\e)=\int_{z_1(\Delta)}^{z_2(\Delta)}\frac{z^n}{H_\e^+(f_\Delta(z))}dz=\int_{z_1(\Delta)}^{z_2(\Delta)}\frac{(1-\e^2f_\Delta(z))z^n}{\sqrt{f_\Delta(z)(2-\e^2f_\Delta(z))}}dz,
\end{equation}
where in the last passage we substitute the formula for $H^+_\e$ in \eqref{eq:defP^-1}.
The above rewriting of the integrals $I_n(\Delta)$ is instrumental in solving system \eqref{eq:I_0,1}, which then leads to the main result of this section.
\begin{thm}\label{thm:Maxwell}
For any $\delta\in\left(0,(\beta_0-\alpha_0)/2\right)$,
there exist $\e_\delta>0$ and a neighborhood $\mathcal{N}_\delta$ of the Maxwell point $\Delta_0$, defined in \eqref{eq:Maxwell-p}, such that
for $\e\in(0,\e_\delta)$ and 
\begin{equation}\label{eq:R_delta}
    r\in\mathcal{R}_\delta :=[\alpha_0+\delta,\beta_0-\delta],    
\end{equation} 
problem \eqref{eq:min} has exactly one simple solution with corresponding admissible pair
$\Delta(\e,r)\in\mathcal{N}_\delta$.
Moreover, there is a constant $C=C_\delta>0$ such that
\begin{equation}\label{eq:O-grande}
	\begin{aligned}
		&\hspace{4cm} \Delta(\e,r)=\Delta_0+O(\exp(-C/\e)), \\ 
		&z_1(\Delta(\e,r))=\alpha_0+O(\exp(-C/\e)), \qquad \quad z_2(\Delta(\e,r))=\beta_0+O(\exp(-C/\e)),
	\end{aligned}
\end{equation}
as $\e\to0^+$, uniformly for $r\in\mathcal{R}_\delta$.
\end{thm}
Hence, the \textbf{Maxwell solution}, mentioned in Section \ref{sec:intro}, is the one whose existence and uniqueness is established by Theorem \ref{thm:Maxwell}.
The Maxwell solution plays a crucial role, because it minimizes the energy functional \eqref{eq:energy} with mass constraint \eqref{eq:constraint}, among all its stationary points, as we shall see in Section \ref{sec:min}.
We prove Theorem \ref{thm:Maxwell} in the next section.
\begin{rem}\label{rem:mass}
Condition \eqref{eq:R_delta} is necessary to prove the existence of simple solutions to \eqref{eq:min}.
Indeed, as we mentioned in Remark \ref{rem:unique}, all possible solutions of \eqref{eq:min} are constants or restriction of periodic solutions in
the real line, and the latter, oscillating between $z_1$ and $z_2$ in \eqref{eq:boundary-z}, have mass in $[2z_1\e^{-1},2z_2\e^{-1}]$. 
As a consequence, if $r\leq\alpha_0$ or $r\geq\beta_0$, then \eqref{eq:min} has the unique solution $u\equiv r$ in view of \eqref{eq:z_i-sigma} below.
Actually, it is possible to prove a stronger result: given $\e>0$, there exists $\delta_\e>0$ such that problem \eqref{eq:min} does not admit simple solutions if $r<\alpha_0+\delta_\e$ or $r>\beta_0-\delta_\e$.
We omit the proof of this fact, which can be easily done by proceeding as in \cite[Theorem 7.1]{CGS} and by taking advantage of the analysis we present in the subsequent Section \ref{sec:proof}.
\end{rem}
\section{Proof of Theorem \ref{thm:Maxwell}}\label{sec:proof}
The strategy is to solve system \eqref{eq:I_0,1} for small $\e$ using the Implicit Function Theorem and thus we shall analyze carefully the behavior of $I_0$ and $I_1$ with respect to this parameter. 
To do this, notice that, by rewriting $1/H_\e^+$ as
$$\frac{1}{H_\e^+(s)}=\frac{1}{\sqrt{2s}}+\frac{\e^2\sqrt{s}}{\sqrt2(2-\e^2s)+2\sqrt{2-\e^2s}}-\frac{\e^2\sqrt{s}}{\sqrt{2-\e^2s}},$$
and substituting in \eqref{eq:I_n}, we deduce
\begin{equation}\label{eq:splitIn}
    I_n(\Delta;\e)=\frac{1}{\sqrt2}\int_{z_1(\Delta)}^{z_2(\Delta)}\frac{z^n}{\sqrt{f_\Delta(z)}}\,dz+\e^2 R_n(\Delta;\e),
\end{equation}
where
\begin{equation}\label{eq:Rn}
    R_n(\Delta;\e):=-\int_{z_1(\Delta)}^{z_2(\Delta)}\frac{\sqrt{f_\Delta(z)}(\sqrt{2(2-\e^2f_\Delta(z))}-1)z^n}{\sqrt{2(2-\e^2f_\Delta(z)}(\sqrt{2-\e^2f_\Delta(z)}+\sqrt2)}\,dz.
\end{equation}
In particular, we shall single out the terms in \eqref{eq:splitIn} which are singular in $\e$, while the remainder should be treated implicitly. 
For this, as the involved functions are not globally $C^1$ close to $\Delta_0$, we need to introduce the following notions as stated in \cite{CGS}.

\begin{rem}\label{rem:generic-Q-4}
The specific form \eqref{eq:Q} for $Q$ again plays a crucial role in \eqref{eq:splitIn}-\eqref{eq:Rn} because it leads to the explicit formula \eqref{eq:defP^-1}.
In the case of a generic function $Q$ satisfying \eqref{eq:Q-generic}, one can use solely the expansion of the inverse $H^+_\e$ close to zero (as already claimed in Remark \ref{rem:unique}, see 
\cite[Lemma 2.1]{DFS}) in the proof of  the needed properties of the remainder $R_n$, without having an explicit formula at our disposal.
\end{rem}
 
Let $\Sigma$ be the set of all admissible pairs $(\sigma,b)$ and let us write the boundary of $\Sigma$ as the disjoint union 
\begin{equation*}
	\partial\Sigma=\bigcup_{i=1}^3\partial_i\Sigma,
\end{equation*}
where   
\begin{align*}
    \partial_1\Sigma&:=\left\{(\sigma,b) \, : \, b=\Phi_\sigma(\alpha_\sigma), \, \sigma\in[\sigma_0,\overline\sigma)\right\},\\
    \partial_2\Sigma&:=\left\{(\sigma,b) \, : \, b=\Phi_\sigma(\beta_\sigma), \, \sigma\in(\underline\sigma,\sigma_0]\right\},\\
    \partial_3\Sigma&:=\left\{(\sigma,b) \, : \, b=\Phi_\sigma(\zeta_\sigma), \, \sigma\in[\underline\sigma,\overline\sigma]\right\}.
\end{align*}
Note that $\partial_1\Sigma$ and $\partial_2\Sigma$ intersect at $\Delta_0=(\sigma_0,b_0)$ and that $\partial\Sigma$ also has cusps at the points $\underline\Delta:=(\underline\sigma,\underline b)$, $\overline\Delta:=(\overline\sigma,\overline b)$, where
\begin{equation*}
    \underline b:=\Phi_{\underline\sigma}(\beta_{\underline\sigma}),
    \qquad \qquad \overline b:=\Phi_{\overline\sigma}(\beta_{\overline\sigma}).
\end{equation*}
The functions $z_i(\Delta)$, $i=1,2$, are continuous in $\overline\Sigma$ and $z_i\in C^4(\Sigma)$; moreover,
\begin{equation}\label{eq:z_i-sigma}
    \alpha_0<z_1(\Delta)<z_2(\Delta)<\beta_0, \qquad \qquad \forall\, \Delta\in\Sigma,
\end{equation}
while in the boundary one has
\begin{equation}\label{eq:z_i-boundary}
    \begin{aligned}
        z_1(\Delta)=\alpha_\sigma, \qquad \qquad \Delta\in\partial_1\Sigma,\\
        z_2(\Delta)=\beta_\sigma, \qquad \qquad \Delta\in\partial_2\Sigma,\\
        z_1(\Delta)=z_2(\Delta)=\zeta_\sigma, \qquad \qquad \Delta\in\partial_3\Sigma.
    \end{aligned}
\end{equation}
Finally,
\begin{equation}\label{eq:z_i-final}
    \begin{aligned}
        \Phi'_\sigma(z_1(\Delta))\neq0, \qquad \mbox{ for } \Delta\in\partial_2\Sigma\backslash\left\{\Delta_0\right\},\\
        \Phi'_\sigma(z_2(\Delta))\neq0, \qquad \mbox{ for } \Delta\in\partial_1\Sigma\backslash\left\{\Delta_0\right\}.
    \end{aligned}
\end{equation}
The proofs of \eqref{eq:z_i-sigma}, \eqref{eq:z_i-boundary}, \eqref{eq:z_i-final} can be found in \cite[Proposition 4.2]{CGS}.
Notice in particular that, in view of \eqref{eq:z_i-sigma}, the assumption $\e\in(0,\overline F)$  guarantees the validity of \eqref{eq:z'=+}, with $f_\Delta(z)$ defined in \eqref{eq:f_delta}, for any $z\in (z_1(\Delta), z_2(\Delta))$, $i=1,2$, and $\Delta \in \Sigma$. 

In order to solve system \eqref{eq:I_0,1}, we introduce the transformation $h=\pi(\Delta)$, defined by
\begin{equation}\label{eq:trans-pi}
    h=(h_1,h_2)=(\pi_1(\Delta),\pi_2(\Delta)), \qquad \mbox{ with }\quad 
    \begin{cases}
    \pi_1(\Delta)=b-\Phi_\sigma(\alpha_\sigma),\\
    \pi_2(\Delta)=b-\Phi_\sigma(\beta_\sigma).
    \end{cases}
\end{equation}
From the definition of the Gibbs function \eqref{eq:Gibbs} and the fact that $\alpha_\sigma$, $\beta_\sigma$ are critical points of $\Phi_\sigma$, 
it follows that the Jacobian determinant of the mapping $\pi$ is
\begin{equation*}
    \mbox{\textrm{det} }J_\pi(\Delta)=\alpha_\sigma-\beta_\sigma<0, \qquad \qquad \forall\,\sigma\in[\underline\sigma,\overline{\sigma}].
\end{equation*}
Hence, $\pi$ is a $C^4$ diffeomorphism of $\overline{\Sigma}$ onto $\overline{\mathcal{H}}$, where $\mathcal{H}=\pi(\Sigma)$.
In particular, the definitions of $\Delta_0$, $\underline\Delta$, $\overline{\Delta}$ give
\begin{equation*}
    \pi(\Delta_0)=0, \qquad \pi_1(\overline\Delta)=0, \qquad \pi_2(\underline\Delta)=0.
\end{equation*}
Finally, for later use, it is worth mentioning that, in view of \eqref{eq:b}, $\pi_i \geq 0$ in $\bar\Sigma$ and $\pi_i > 0$ in $\Sigma$, $i=1$, $2$.
\begin{defn}[\cite{CGS}]\label{def:NR}
We say that a function $\psi:=\psi(h)$ is \textbf{nearly regular at} $h=0$ if there exists a neighborhood $\mathcal{U}\subset\mathcal{H}$ of $0$ and $s<1$  such that 
$\psi\in C(\mathcal{U})$, $\psi\in C^1(\mathcal{U}\backslash\{0\})$, and
\begin{align*}
    \psi(h)&=\psi(0)+\mathcal{O}\left(|h|^{1-s}\right),\\
    \nabla\psi(h)&=\left(\mathcal{O}\left(h_1^{-s}\right),\mathcal{O}\left(h_2^{-s}\right)\right).
\end{align*} 
\end{defn}
In the sequel, we will use the following notations: 
given a function $\varphi:\Sigma\to\R$, we write 
\begin{equation*}
    \hat{\varphi}(h)=\varphi(\pi^{-1}(h)).
\end{equation*}
Moreover, if $\Omega\subset\R^2$, then we denote
\begin{equation*}
    C^+(\Omega):=\left\{f\in C(\Omega)\, : \, f(x)>0\, \mbox{ for any } x\in\Omega\right\}.
\end{equation*}
Finally, if $\Omega$ is a neighborhood of $\Delta_0$ in $\overline\Sigma$, we define
\begin{align*}
    NR(\Omega)&:=\left\{\varphi\in C(\Omega)\, : \, \hat{\varphi}(h) \mbox{ is nearly regular at } h=0\right\},\\
    NR^+(\Omega)&:=\left\{\varphi\in NR(\Omega) \, : \, \varphi(x)>0\, \mbox{ for any } x\in\Omega\right\}. 
\end{align*}
Let 
\begin{equation*}
    \mathcal{F}_i:=\left\{\Delta\in\overline{\Sigma}\, : \, F''(z_i(\Delta))>0\right\},
\end{equation*}
and let $\Omega_0$, $\Omega_1$ and $\Omega_2$  be arbitrary closed subregions of $\overline{\Sigma}$ with
\begin{equation*}
    \Omega_0\subset\mathcal{F}_1\cap\mathcal{F}_2, \qquad 
    \Omega_1\subset \left(\mathcal{F}_1\backslash\partial_2\Sigma\right), \qquad \qquad     \Omega_2\subset\left(\mathcal{F}_2\backslash\partial_1\Sigma\right).
\end{equation*}
so that
\begin{equation}\label{eq:F''>0}
    F''(z_i(\Delta))\geq C>0 \qquad \mbox{ on } \Omega_0\cup\Omega_i, \qquad \qquad i=1,2.
\end{equation}
Notice that in the definition \eqref{eq:splitIn} of $I_n$ the first integral has an integrand which is singular at $z=z_i(\Delta)$, because $f_\Delta(z_i(\Delta))=0$, while $R_n$ is continuous as a function of $\Delta$ in $\bar\Sigma$, thanks to \eqref{eq:Rn} and Lemma \ref{lem:fdelta}.
Therefore, the most relevant case in our analysis is when $\Delta_0\in\Omega_0$, that is where the integrals $I_0, I_1$ defined in \eqref{eq:I_n} blow up, and we shall use the notion introduced in Definition \ref{def:NR}.
Our fundamental result which is instrumental to solve system \eqref{eq:I_0,1} close to $\Delta = \Delta_0$, or, equivalently, close to $h=0$, is contained in the next proposition and concerns the precise nature of the singularities of $I_0$, $I_1$ at the Maxwell point $\Delta_0$. 
With this aim in mind, let us first recall that $z_i(\Delta)$ belong to $NR(\bar\Sigma)$ and for $h=\pi(\Delta)$, one has
\begin{equation}\label{eq:z_i-NR}
    \begin{aligned}
        z_1(\Delta)&=\alpha_\sigma+ \sqrt{h_1}\Theta_1(\Delta),\\
        z_2(\Delta)&=\beta_\sigma- \sqrt{h_2}\Theta_2(\Delta),
    \end{aligned}
\end{equation}
on $\Sigma$, with $\Theta_i\in NR^+(\bar\Sigma)\backslash\left\{\underline\Delta,\overline{\Delta}\right\}$.
For the proof of \eqref{eq:z_i-NR}, see \cite[Proposition 4.3]{CGS}.
Next, let us introduce the functions $\mathcal Q_1$ and $\mathcal Q_2$ as follows:
\begin{equation}\label{eq:Q_i}
    \mathcal Q_i(\Delta):=\left[2F''(z_i(\Delta))\right]^{-\frac12}.
\end{equation}
It is worth mentioning that \eqref{eq:F''>0} and \cite[Lemma 4.2]{CGS} imply
\begin{equation}\label{eq:Q_i-NR}
     \mathcal Q_i\in NR^+(\Omega_i\cup\Omega_0), \qquad\qquad i=1,2.
\end{equation}
\begin{prop}\label{prop:In}
Assume $F$ satisfies \eqref{eq:hypF} and $\e\in(0,\overline F)$, where $\overline F$ is defined in \eqref{eq:eps-small-final}.
Then, $I_n$ defined in \eqref{eq:I_n} verifies
\begin{align}
    &I_n + \frac{1}{\sqrt{2}}\mathcal Q_1z_1^n \ln(\pi_1)\in C(\Omega_1),  \qquad I_n + \frac{1}{\sqrt{2}}\mathcal Q_2z_2^n \ln(\pi_1)\in C(\Omega_2), \qquad n=0,1, \label{eq:InCont}\\
    &I_n + \frac{1}{\sqrt{2}}\mathcal Q_1z_1^n \ln(\pi_1) + \frac{1}{\sqrt{2}}\mathcal Q_2z_2^n  \ln(\pi_2) \in NR(\Omega_0), \qquad n=0,1. \label{eq:InNR}
\end{align}
\end{prop}
In view of \cite[Proposition 5.1]{CGS}, \eqref{eq:splitIn}, \eqref{eq:Rn}, and the continuity of $R_0(\Delta;\e)$ and $R_1(\Delta;\e)$ in $\Omega_0\cup\Omega_1\cup\Omega_2$,
the above result is a direct consequence of the following lemma.
\begin{lem}\label{Lem:Rn}
Under the same assumptions of Proposition \ref{prop:In}, $R_0$ and $R_1$ defined as in \eqref{eq:Rn} belong to $NR(\Omega_0)$.
\end{lem}
\begin{proof}
From \eqref{eq:Rn}, it follows that
\begin{align*}
    R_0(\Delta;\e)=-\int_{z_1(\Delta)}^{z_2(\Delta)}\frac{\sqrt{f_\Delta(z)}(\sqrt{2(2-\e^2f_\Delta(z))}-1)}{\sqrt{2(2-\e^2f_\Delta(z)}(\sqrt{2-\e^2f_\Delta(z)}+\sqrt2)}\,dz, \\
    R_1(\Delta;\e)=-\int_{z_1(\Delta)}^{z_2(\Delta)}\frac{\sqrt{f_\Delta(z)}(\sqrt{2(2-\e^2f_\Delta(z))}-1)z}{\sqrt{2(2-\e^2f_\Delta(z)}(\sqrt{2-\e^2f_\Delta(z)}+\sqrt2)}\,dz.
\end{align*}
We need to control the above integrals only close to the boundary points $z_1(\Delta)$ and $z_2(\Delta)$, namely where $f_\Delta$ vanishes; see \eqref{eq:f_delta}.
It is worth observing that, contrary to $I_0$ and $I_1$, these integrals will not blow up, and actually they are continuous in the whole region $\bar \Sigma$, at least for $\e$ smaller than $\overline F$ defined in \eqref{eq:eps-small-final}.  However,  they are not smooth close to the boundary points,   yet we shall be able to prove they are nearly regular. For this, we focus our attention in the region $(z_1(\Delta), z_1(\Delta) + \eta)$ for a fixed $\eta>0$ sufficiently small; the analysis close to $z_2(\Delta)$ being analogous. 

Let us start with the study of  $R_0$, and define the following integral
\begin{equation*}
     R^\eta_0(\Delta;\e) := \int_{z_1(\Delta)}^{z_1(\Delta)+\eta}\sqrt{f_\Delta(z)}G_\e(f_\Delta(z))\,dz = \int_{0}^{\eta}\sqrt{f_\Delta(z_1(\Delta)+t)}G_\e(f_\Delta(z_1(\Delta)+t))\,dt,
\end{equation*}
where
\begin{equation*}
     G_\e(s) := -\frac{\sqrt{2(2-\e^2s)}-1}{\sqrt{2(2-\e^2s)}(\sqrt{2-\e^2s}+\sqrt2)}.
\end{equation*}
Notice that the function $G_\e\in C^\infty(0,2\e^{-2})$ and 
\begin{equation*}
    G_\e(s)=G_\e(0)+\tilde G_\e(s)=-\frac1{4\sqrt2}+\frac{s+3\sqrt{2(2-\e^2s)}-6}{4\sqrt2\sqrt{2-\e^2s}(\sqrt{2-\e^2s}+\sqrt2)}, \qquad s\in(0,2\e^{-2}).
\end{equation*}
It follows that
\begin{equation*}
    R^\eta_0(\Delta;\e)= -\frac{1}{4\sqrt2} \int_{0}^{\eta}\sqrt{f_\Delta(z_1(\Delta)+t)}\,dt+\int_{0}^{\eta}\sqrt{f_\Delta(z_1(\Delta)+t)}\tilde G_\e(f_\Delta(z_1(\Delta)+t))\,dt,
\end{equation*}
and since the function $\sqrt s\tilde G_\e(s)\in C^1(0,2\e^{-2})$, $R^\eta_0(\Delta;\e)\in NR(\Omega_0)$ if and only if $\tilde R^\eta_0(\Delta;\e)\in NR(\Omega_0)$, where
\begin{equation}\label{eq:tilde R_0} 
    \tilde R^\eta_0(\Delta) :=\int_{0}^{\eta}\sqrt{f_\Delta(z_1(\Delta)+t)}\,dt.
\end{equation}
In order to study the latter integral, we expand $f_\Delta(z_1(\Delta)+t)$ about $t=0$ using Taylor's formula: first, we have
\begin{equation*}
    f_\Delta(z_1(\Delta)+t)=\Phi'_\sigma(z_1(\Delta))t+ \frac{1}{2}F''(z_1(\Delta))t^2+ \frac16F'''(z_1(\Delta))t^3+\Theta_F(z_1(\Delta),t)t^4,
\end{equation*}
where $\Theta_F(y,t)$ is a $C^1$ function for $y>0$ and we used $f(z_1(\Delta))=0$ and the formulas \eqref{eq:f_delta}-\eqref{eq:Gibbs}.
As a consequence, we can rewrite
\begin{equation*}
    f_\Delta(z_1(\Delta)+t)=\left[2\mathcal Q_{1}(\Delta)\right]^{-2}R(t,\Delta),
\end{equation*}
with $\mathcal{Q}_1$ defined in \eqref{eq:Q_i} and 
\begin{equation}\label{eq:deflambda}
\begin{aligned}
    R(t,\Delta)&:=2\lambda(\Delta)t+t^2+c(\Delta)t^3+ \rho(t,\Delta)t^4, \\
    \lambda(\Delta)&:=\frac{\left[F'(z_1(\Delta))-F'(\alpha_0)\right]}{F''(z_1(\Delta))},\\
    c(\Delta)&:=\frac{F'''(z_1(\Delta))}{3F''(z_1(\Delta))}, \qquad \qquad  
    \rho(t,\Delta):=\frac{2\Theta_F(z_1(\Delta),t)}{F''(z_1(\Delta))}. 
\end{aligned}
\end{equation}
In particular, $\lambda, c\in C(\Omega_0)$, $\rho\in C([0,\infty)\times\Omega_0)$ and $\lambda>0$ on $\Omega_0\backslash\partial_1\Sigma$ because of \eqref{eq:F''>0}.
Moreover, one has
\begin{equation*}
    R(t,\Delta)=\left[2\lambda(\Delta)t+t^2\right]
    \left[1+\frac{c(\Delta)t^3+ \rho(t,\Delta)t^4}{2\lambda(\Delta)t+t^2}\right]=: \left[2\lambda(\Delta)t+t^2\right]
    \left[1+K(t,\Delta)\right],
\end{equation*}
and substituting in \eqref{eq:tilde R_0}, we deduce
\begin{align*}
     \tilde R^\eta_0(\Delta)&=\left[2\mathcal Q_{1}(\Delta)\right]^{-1}\int_{0}^{\eta}R(t,\Delta)^{1/2}\,dt\\
     &=\left[2\mathcal Q_{1}(\Delta)\right]^{-1}\int_{0}^{\eta}\left[2\lambda(\Delta)t+t^2\right]^{1/2}
    \left[1+K(t,\Delta)\right]^{1/2}\,dt.
\end{align*}
By using the binomial formula
\begin{equation*}
    \left[1+K(t,\Delta)\right]^{1/2}=1+\frac{1}2K(t,\Delta)
    +\sum_{m=2}^{\infty}a_m K(t,\Delta)^m,
\end{equation*}
with $a_m$ appropriate binomial coefficients, we infer
\begin{align*}
     \tilde R^\eta_0(\Delta)&=\left[2\mathcal Q_{1}(\Delta)\right]^{-1}\int_{0}^{\eta}\left[2\lambda(\Delta)t+t^2\right]^{1/2}\,dt\\
     &\qquad +\left[4\mathcal Q_{1}(\Delta)\right]^{-1}c(\Delta)\int_{0}^{\eta}\frac{t^3}{\left[2\lambda(\Delta)t+t^2\right]^{1/2}}\,dt\\
      &\qquad +\left[4\mathcal Q_{1}(\Delta)\right]^{-1}\int_{0}^{\eta}\frac{ \rho(t,\Delta)t^4}{\left[2\lambda(\Delta)t+t^2\right]^{1/2}}\,dt\\
      &\qquad +\left[2\mathcal Q_{1}(\Delta)\right]^{-1}\int_{0}^{\eta}\left[2\lambda(\Delta)t+t^2\right]^{1/2}\sum_{m=2}^{\infty}a_m K(t,\Delta)^m\,dt\\
      &=:L_1(\Delta)+L_2(\Delta)+L_3(\Delta)+L_4(\Delta).
\end{align*}
Let us analyze all the contributions $L_i(\Delta)$, $i=1,2,3,4$ and,  being explicitly computed, we carried out in full details the first two terms. For this,  
routine integrations show that
\begin{align*}
    \int_{0}^{\eta} \left[2\lambda(\Delta)t+t^2\right]^{1/2}\,dt&=\lambda(\Delta)\int_{0}^{\eta}\sqrt{\left(\frac{t+\lambda(\Delta)}{\lambda(\Delta)}\right)^2-1}\,dt\\
    &=\lambda(\Delta)^2\int_1^{1+\frac{\eta}{\lambda(\Delta)}}\sqrt{y^2-1}\, dy\\
    &=\frac{\lambda(\Delta)^2}{2}\left[\left(1+\frac{\eta}{\lambda(\Delta)}\right)\sqrt{\frac{\eta^2}{\lambda(\Delta)^2}+\frac{2\eta}{\lambda(\Delta)}}\right.\\
    &\qquad \left.-\ln\left(\sqrt{\frac{\eta^2}{\lambda(\Delta)^2}+\frac{2\eta}{\lambda(\Delta)}}+1+\frac{\eta}{\lambda(\Delta)}\right)\right]\\
    &=\frac{1}{2}\left[\left(\lambda(\Delta)+\eta\right)\sqrt{\eta^2+2\eta\lambda(\Delta)}\right]+\frac{\lambda(\Delta)^2}{2}\ln\left(\lambda(\Delta)\right)\\
    &\qquad -\frac{\lambda(\Delta)^2}{2}\ln\left(\sqrt{\eta^2+2\eta\lambda(\Delta)}+\lambda(\Delta)+\eta\right),
\end{align*}
and 
\begin{align*}
    \int_{0}^{\eta}\frac{t^3}{\left[2\lambda(\Delta)t+t^2\right]^{1/2}}\,dt&= -5\lambda(\Delta)^3\ln\left(\sqrt{\frac{\eta}{2\lambda(\Delta)}}+\sqrt{\frac{\eta}{2\lambda(\Delta)}+1}\right)\\
    &\qquad +\frac13\sqrt{\eta^5(\eta+2\lambda(\Delta))}
    -\frac{5}6\lambda(\Delta)(\eta-3\lambda(\Delta))\sqrt{\eta^2+2\eta\lambda(\Delta)}\\
    &=-\frac{5\lambda(\Delta)^3}{2}\ln\left(\eta+\lambda(\Delta)+\sqrt{\eta^2+2\eta\lambda(\Delta)}\right)+\frac{5\lambda(\Delta)^3}{2}\ln(\lambda(\Delta))\\
    &\qquad +\frac13\sqrt{\eta^5(\eta+2\lambda(\Delta))}-\frac{5}6\lambda(\Delta)(\eta-3\lambda(\Delta))\sqrt{\eta^2+2\eta\lambda(\Delta)}.
\end{align*}
Then, since $\lambda(\Delta)\in NR(\Omega_0)$ and $\lambda(\Delta)\ln\lambda(\Delta)\in NR(\Omega_0)$ (see \cite[Lemma 5.1]{CGS}), we can state that $L_i(\Delta)\in NR(\Omega_0)$, for $i=1,2$, that is the ``leading terms'' in $\tilde R_0^\eta(\Delta)$.
The analysis of $L_i(\Delta)$, for $i=3,4$ is straightforward, in view of the arguments carried out in \cite{CGS}. Indeed, the integrand in these terms can be recast as the corresponding ones of \cite{CGS} multiplied by $\left[2\lambda(\Delta)t+t^2\right]$ and thus they are more regular at $t=0$ than the ones already discussed there.  
Hence, we can conclude that $\tilde R_0^\eta(\Delta)\in NR(\Omega_0)$.

Concerning $R_1(\Delta)$, we can proceed in the same way and we arrive to study
\begin{equation*}
    \tilde R^\eta_1(\Delta):=\int_{0}^{\eta}(z_1(\Delta)+t)\sqrt{f_\Delta(z_1(\Delta)+t)}\,dt=
    z_1(\Delta)\tilde R^\eta_0(\Delta)+\int_{0}^{\eta}t\sqrt{f_\Delta(z_1(\Delta)+t)}\,dt.
\end{equation*}
Therefore, we can say that $\tilde R^\eta_1(\Delta)\in NR(\Omega_0)$ because $z_i(\Delta)$ belong to $NR(\bar\Sigma)$ (cf. \eqref{eq:z_i-NR}) and the last integral can be treated exactly as in the proof of $\tilde R^\eta_0(\Delta)\in NR(\Omega_0)$.
In conclusion, $R_0$ and $R_1$ belong to $NR(\Omega_0)$ and the proof is complete.
\end{proof}
\begin{proof}[Proof of Proposition \ref{prop:In}]
In view of \cite[Proposition 5.1]{CGS}, we readily obtain
\begin{align*}
    \frac{1}{\sqrt2}\int_{z_1}^{z_2}\frac{z^n}{\sqrt{f_\Delta(z)}}\,dz + \frac{1}{\sqrt{2}}\mathcal Q_1z_1^n \ln(\pi_1)\in C(\Omega_1), \\
    \frac{1}{\sqrt2}\int_{z_1}^{z_2}\frac{z^n}{\sqrt{f_\Delta(z)}}\,dz + \frac{1}{\sqrt{2}}\mathcal Q_2z_2^n \ln(\pi_1)\in C(\Omega_2),
\end{align*}
for $n=0,1$ and
\begin{align*}
    \frac{1}{\sqrt2}\int_{z_1}^{z_2}\frac{1}{\sqrt{f_\Delta(z)}}\,dz  
 + \frac{1}{\sqrt{2}}\mathcal Q_1 \ln(\pi_1) + \frac{1}{\sqrt{2}} \mathcal Q_2  \ln(\pi_2) = \tilde{\Psi}_0 \in NR(\Omega_0),
 \\
 \frac{1}{\sqrt2}\int_{z_1}^{z_2}\frac{z}{\sqrt{f_\Delta(z)}}\,dz+ \frac{1}{\sqrt{2}}\mathcal Q_1z_1  \ln(\pi_1) + \frac{1}{\sqrt{2}}\mathcal Q_2z_2  \ln(\pi_2)
 = \tilde{\Psi}_1 \in NR(\Omega_0).
\end{align*}
Then, the continuity of $R_n$ implies \eqref{eq:InCont}, while Lemma \ref{Lem:Rn} gives
\begin{align}
    I_0  
 + \frac{1}{\sqrt{2}}\mathcal Q_1 \ln(\pi_1) + \frac{1}{\sqrt{2}} \mathcal Q_2  \ln(\pi_2) = {\Psi}_0 \in NR(\Omega_0),
 \label{eq:condiI0_NR}\\
 I_1 + \frac{1}{\sqrt{2}}\mathcal Q_1z_1  \ln(\pi_1) + \frac{1}{\sqrt{2}}\mathcal Q_2z_2  \ln(\pi_2)
 = {\Psi}_1 \in NR(\Omega_0),
 \label{eq:condiI1_NR}
\end{align}
with
\begin{equation}\label{eq:Psin}
    {\Psi}_0(\Delta;\e) := \tilde{\Psi}_0(\Delta) + \e^2R_0(\Delta;\e), \qquad 
    {\Psi}_1(\Delta;\e) := \tilde{\Psi}_1(\Delta) + \e^2R_1(\Delta;\e)
\end{equation}
and the proof is complete.
\end{proof}
\begin{rem}\label{rem:o1o2}
As already mentioned before, the regularity results \eqref{eq:InNR}  proved in      in Proposition \ref{prop:In} refers solely to the set $\Omega_0$, as they are instrumental to  prove the existence and uniqueness of Maxwell solution close to 
$\Delta_0\in\Omega_0$. 
However, one could split the analysis of the integrals $I_n$ close to the two points $z_1(\Delta)$ and $z_2(\Delta)$ separately, and give more refined regularity results there, as done for instance in \eqref{eq:Q_i-NR}, see \cite{CGS} for further details.
Analogously, a more refined version (again involving   $\Omega_1$ and $\Omega_2$) of the results contained in Lemma \ref{Lem:Rn} can also be proved. 
Finally, it is worth mentioning that the properties \eqref{eq:InCont} are not needed at this stage,  but they will be crucial afterwards, while studying the energy of other (non-Maxwell) simple solutions (if any), which are defined away from $\Delta_0$; see Proposition \ref{prop:no-Maxw} and Proposition \ref{prop:less} below. 
\end{rem}
After studying the behavior of the functions $I_n(\Delta;\e)$, the next step of our proof of Theorem \ref{thm:Maxwell} will consist in solving system \eqref{eq:I_0,1} in a neighborhood of the Maxwell point $\Delta_0$, see \eqref{eq:Maxwell-p}.
With this in mind and following \cite{CGS}, we introduce the scaling
\begin{equation}\label{eq:h_i}
    h_i=\exp\left(\mu_i k_i-\frac{c_i(r)}{\e}\right), \qquad \qquad i=1,2,
\end{equation}
where
\begin{align}
    c_1(r)&:=\frac{2\sqrt2(\beta_0-r)}{B_1(\beta_0-\alpha_0)}, \qquad \qquad 
    B_1:=\left[2F''(\alpha_0)\right]^{-\frac12}, \label{eq:c1-B1}\\
    c_2(r)&:=\frac{2\sqrt2(r-\alpha_0)}{B_2(\beta_0-\alpha_0)}, \qquad \qquad 
    B_2:=\left[2F''(\beta_0)\right]^{-\frac12}, \label{eq:c2-B2}\\
    \mu_i&:=\left[B_i(\beta_0-\alpha_0)\right]^{-1}. \qquad \qquad i=1,2. \label{eq:mu_i}
\end{align}
The transformation \eqref{eq:h_i} is defined for $k=(k_1,k_2)\in\R^2$, $\e>0$ and $r\in(\alpha_0,\beta_0)$.
If $(k,\e,r)$ is such that $h=(h_1, h_2)$, defined by \eqref{eq:h_i}, belongs to $\mathcal{H}=\pi(\Sigma)$, where the transformation $\pi$ is given by \eqref{eq:trans-pi}, then we call $(k,\e,r)$ \textbf{compatible} and $\Delta=\pi^{-1}(h)$ is an admissible pair; let us denote the resulting map as
\begin{equation*}
    \Delta=D(k,\e,r).
\end{equation*}
Conversely, given any admissible pair $\Delta$ and any $\e>0$, we can define $h=\pi(\Delta)$ and solve \eqref{eq:h_i} for $k$; we denote this mapping by
\begin{equation*}
    k=K(\Delta,\e,r).
\end{equation*}
Given any function $\varphi:=\varphi(\Delta)$, we use the notation
\begin{equation*}
    \varphi^*(k,\e,r):=\varphi(D(k,\e,r)).
\end{equation*}
In order to study implicitly \eqref{eq:I_0,1} close to the Maxwell point and for sufficiently small $\e>0$, 
we need to extend smoothly functions  $\varphi^*$ also for negative values of $\e$, namely to a neighborhood in $\R^4$ of the set
\begin{equation*}
    \mathcal{E}_0:=\left\{(k,\e,r) \, : \, k\in\R^2, \,  \e=0, \, r\in(\alpha_0,\beta_0)\right\}.
\end{equation*}
The needed extension is guaranteed  by the following lemma, proved in \cite{CGS}, and where the regularity properties introduced in Definition \ref{def:NR}   play a crucial role. 
We say that a neighborhood $\mathcal{K}$ of $\mathcal{E}_0$ is compatible if every $(k,\e,r)\in\mathcal{K}$, with $\e>0$, is compatible. 
Finally, given any $C^1$ function $\Psi=\Psi(k,\e,r)$, we write
\begin{equation*}
    \nabla\Psi=\left(\frac{\partial\Psi}{\partial k_1}, \frac{\partial\Psi}{\partial k_2}, \frac{\partial\Psi}{\partial\e}, \frac{\partial\Psi}{\partial r}\right).
\end{equation*} 
\begin{lem}\label{lem:CGS-ext}
Let
\begin{equation*}
    \varphi=\varphi_0+\varphi_1\ln\pi_1+\varphi_2\ln\pi_2,
\end{equation*}
with
\begin{equation*}
    \varphi_0,\varphi_1,\varphi_2\in NR(\Omega_0), \qquad \qquad 
    \varphi_1(\Delta_0)=\varphi_2(\Delta_0)=0.
\end{equation*}
Then, there exists a compatible neighborhood $\mathcal{K}$ of $\mathcal{E}_0$ such that
\begin{itemize}
    \item $\varphi^*$ has a $C^1$ extension to $\mathcal{K}$; 
    \item $\nabla\varphi^*=0$ on $\mathcal{E}_0$.
\end{itemize}
\end{lem}
For the proof of Lemma \ref{lem:CGS-ext} see \cite[Lemma 6.1]{CGS}. 
Thanks to Lemma \ref{lem:CGS-ext} we can then prove the following result.

\begin{prop}\label{prop:pre-implicit}
There exist a compatible neighborhood $\mathcal{K}$ of $\mathcal{E}_0$ and a $C^1$ function $W=(W_1,W_2)$ on $\mathcal{K}$, with
\begin{equation}\label{eq:gradW}
    \nabla W_i=0, \qquad \qquad \mbox{ on } \mathcal{E}_0,
\end{equation}
such that, given $(k,\e,r)\in\mathcal{K}$ with $\e>0$ and $\Delta=D(k,\e,r)$, one has
\begin{equation}\label{eq:k-W}
    k=W(k,\e,r)  
\end{equation}
if and only if $\Delta$ is a solution of \eqref{eq:I_0,1}.
\end{prop}
\begin{proof}
By the definitions \eqref{eq:c1-B1}-\eqref{eq:c2-B2} of $B_i$ and \eqref{eq:Q_i}, we have
\begin{equation*}
    \mathcal{Q}_i(\Delta_0)=B_i, \qquad \qquad i=1,2.
\end{equation*}
Hence, letting $\xi_1=\alpha_0$, $\xi_2=\beta_0$, thanks to \eqref{eq:Q_i-NR} and the fact that $z_i(\Delta)$ belong to $NR(\bar\Sigma)$ (because of \eqref{eq:z_i-NR}), we end up with
\begin{equation}\label{eq:p-q}
    \mathcal{Q}_i(\Delta)=B_i+q_i(\Delta), \qquad \qquad \mathcal{Q}_i(\Delta)z_i(\Delta)=\xi_iB_i+p_i(\Delta),
\end{equation}
for $i=1,2$, where $p_i,q_i\in NR(\Omega_0)$ and $p_i(\Delta_0)=q_i(\Delta_0)=0$.
Thanks to Proposition \ref{prop:In},  in particular  \eqref{eq:condiI0_NR}-\eqref{eq:condiI1_NR}, and 
 using \eqref{eq:p-q}, we obtain
\begin{align*}
    I_0+\frac{1}{\sqrt{2}}B_1\ln(\pi_1)+\frac{1}{\sqrt{2}}B_2\ln(\pi_2)&=\Psi_0-\frac{1}{\sqrt2}q_1\ln(\pi_1)-\frac{1}{\sqrt2}q_2\ln(\pi_2),\\
    I_1+\frac{1}{\sqrt{2}}\alpha_0B_1\ln(\pi_1)+\frac{1}{\sqrt{2}}\beta_0 B_2\ln(\pi_2)&=\Psi_1-\frac{1}{\sqrt2}p_1\ln(\pi_1)-\frac{1}{\sqrt2}p_2\ln(\pi_2),
\end{align*}
with $\Psi_i \in NR(\Omega_0)$, $i=0,1$.
Then, we can apply Lemma \ref{lem:CGS-ext} to the right hand side of the above  equalities   and deduce that there exist a compatible neighborhood $\mathcal{K}$ of $\mathcal{E}_0$ and functions $S_0, S_1\in C^1(\mathcal{K})$, with $\nabla S_0=\nabla S_1=0$ on $\mathcal{E}_0$ such that
\begin{align*}
    S_0(k,\e,r)&=I_0^*(k,\e,r)+\frac{1}{\sqrt2}B_1\ln(h_1)+\frac{1}{\sqrt2}B_2\ln(h_2),\\
     S_1(k,\e,r)&=I_1^*(k,\e,r)+\frac{1}{\sqrt2}\alpha_0B_1\ln(h_1)+\frac{1}{\sqrt2}\beta_0B_2\ln(h_2).
\end{align*}
Substituting \eqref{eq:h_i}-\eqref{eq:c1-B1}-\eqref{eq:c2-B2}-\eqref{eq:mu_i} and multiplying by $\sqrt2(\beta_0-\alpha_0)$, we conclude
\begin{align*}
    \sqrt2(\beta_0-\alpha_0)I_0^*(k,\e,r)&=-k_1-k_2+2\sqrt2(\beta_0-\alpha_0)\e^{-1}+\sqrt2(\beta_0-\alpha_0)S_0(k,\e,r),\\
    \sqrt2(\beta_0-\alpha_0)I_1^*(k,\e,r)&=-\alpha_0k_1-\beta_0k_2+
    2\sqrt2(\beta_0-\alpha_0)r\e^{-1}+\sqrt2(\beta_0-\alpha_0)S_1(k,\e,r),
\end{align*}
for all $(k,\e,r)\in\mathcal{K}$, with $\e>0$.
Thus, system \eqref{eq:I_0,1} is equivalent to 
\begin{equation*}
    \begin{cases}
        k_1+k_2=\sqrt2(\beta_0-\alpha_0)S_0(k,\e,r),\\
        \alpha_0k_1+\beta_0k_2=\sqrt2(\beta_0-\alpha_0)S_1(k,\e,r),
    \end{cases}
\end{equation*}
which readily gives
\begin{equation*}
    k_1=\sqrt2\beta_0S_0(k,\e,r)-\sqrt2S_1(k,\e,r), \qquad 
    k_2=\sqrt2S_1(k,\e,r)-\sqrt2\alpha_0S_0(k,\e,r).
\end{equation*}
Therefore, \eqref{eq:gradW} and \eqref{eq:k-W} hold true with 
\begin{align*}
    W_1(k,\e,r)&=\sqrt2\beta_0S_0(k,\e,r)-\sqrt2S_1(k,\e,r),\\
    W_2(k,\e,r)&=\sqrt2S_1(k,\e,r)-\sqrt2\alpha_0S_0(k,\e,r),
\end{align*}
and the proof is complete.
\end{proof}
Thanks to \eqref{eq:gradW} and the Implicit Function Theorem we can state that there exists a unique solution $k=k(\e,r)$ of \eqref{eq:k-W} in a neighborhood of a point $(a,0,r)$, where $a\in\R^2$ is fixed and for any $r$ satisfying \eqref{eq:R_delta}.
Indeed, \eqref{eq:gradW} implies 
\begin{equation}\label{eq:a_i}
    W_i(k,0,r)=a_i, \qquad \qquad \forall\, k\in\R^2, \, r\in(\alpha_0,\beta_0),
\end{equation}
for some $a=(a_1,a_2)\in\R^2$.
Hence, $(a,0,r)$ is a solution of \eqref{eq:k-W} for any $r\in(\alpha_0,\beta_0)$.
Moreover, \eqref{eq:gradW} also gives
\begin{equation*}
    \frac{\partial\tilde W}{\partial k}(a,0,r)=\mathbb{I}_2, \qquad \qquad \forall\, r\in(\alpha_0,\beta_0),
\end{equation*}
where $\tilde W(k,\e,r)=k-W(k,\e,r)$ and $\mathbb{I}_2$ is the identity matrix in $\R^2$, and a simple application of the Implicit Function Theorem gives the following result.
\begin{prop}\label{prop:implicit}
Choose $\delta\in\left(0,(\beta_0-\alpha_0)/2\right)$ and define $\mathcal{R}_\delta$ as in \eqref{eq:R_delta}. 
Moreover let $a$ be given by \eqref{eq:a_i} and let $\mathcal{K}$ be the neighborhood of $\mathcal{E}_0$ given by Proposition \ref{prop:pre-implicit}.
Then, there exist a neighborhood $\mathcal{B}_\delta$ of $a\in\R^2$ and $\e_\delta>0$ with $\mathcal{B}_\delta\times[-\e_\delta,\e_\delta]\times\mathcal{R}_\delta\subset\mathcal{K}$ and a $C^1$ function 
\begin{equation*}
    k_\delta : [-\e_\delta,\e_\delta]\times\mathcal{R}_\delta\to\mathcal{B}_\delta, 
\end{equation*}
such that for any $(\e,r)\in[-\e_\delta,\e_\delta]\times\mathcal{R}_\delta$, equation \eqref{eq:k-W} has exactly one solution $k=k_\delta(\e,r)\in\mathcal{B}_\delta$.
\end{prop}
Now, we have all the tools to prove Theorem \ref{thm:Maxwell}.
\begin{proof}[Proof of Theorem \ref{thm:Maxwell}]
Let $k_\delta$ be the solution given by Proposition \ref{prop:implicit}; then, by using Proposition \ref{prop:pre-implicit} we conclude that $\Delta(\e,r)=D(k_\delta(\e,r),\e,r)$ solves system \eqref{eq:I_0,1}.
Moreover, since $k_\delta$ is bounded, \eqref{eq:trans-pi} and \eqref{eq:h_i} imply
\begin{equation}\label{eq:Delta_0}
    h=\pi(\Delta(\e,r))=\mathcal{O}(\exp(-C/\e)), \qquad \qquad \mbox{ as } \e\to0,
\end{equation}
uniformly for $r\in\mathcal{R}_\delta$. Hence, we have the first estimate in \eqref{eq:O-grande} because $\pi$ is diffeomorphic on $\bar\Sigma$ with $\pi(\Delta_0)=0$.
The remaining estimates follow from \eqref{eq:z_i-NR}, \eqref{eq:Delta_0} and the fact that, since $\alpha_\sigma$ and $\beta_\sigma$ are smooth functions of $\sigma$ near $\sigma_0$, one has
\begin{equation*}
    \alpha_\sigma=\alpha_0+\mathcal{O}(\exp(-C/\e)), \qquad \qquad \beta_\sigma=\beta_0+\mathcal{O}(\exp(-C/\e)).
\end{equation*}

It remains to prove the uniqueness of the solution to system \eqref{eq:I_0,1}.
To do this, we will take advantage of the uniqueness of the solution to \eqref{eq:k-W} given by Proposition \ref{prop:implicit}.
Let us use the notation 
\begin{equation*}
    \mbox{Sol}(\e,r):=\left\{\Delta\in\Sigma \, : \, \Delta \, \mbox{ is a solution of \eqref{eq:I_0,1}}\right\}.
\end{equation*}
For any $\Delta\in\Sigma$, \eqref{eq:z_i-sigma} holds true and, in particular, from \eqref{eq:f_delta}-\eqref{eq:fdelta_bounds} it follows that $f_\Delta(z)>0$ for any $z\in(z_1(\Delta),z_2(\Delta))$.
Thus, \eqref{eq:I_n} gives
\begin{equation*}
    z_1(\Delta) I_0(\Delta)<I_1(\Delta)<z_2(\Delta) I_0(\Delta),
\end{equation*}
and if $\Delta\in\mbox{Sol}(\e,r)$, then applying \eqref{eq:I_0,1} we obtain
\begin{equation}\label{eq:r-in-z}
    z_1(\Delta)<r<z_2(\Delta), \qquad \qquad \hbox{for any}\ \Delta\in\mbox{Sol}(\e,r).
\end{equation}
By using  \eqref{eq:condiI0_NR}-\eqref{eq:condiI1_NR}, since $z_i\in C(\Omega_0)$ and $\mathcal{Q}_i$, $z_2-z_1\in C^+(\Omega_0)$, we can solve \eqref{eq:I_0,1} to get
\begin{equation}\label{eq:ln-pi}
    \ln(\pi_i)=-\e^{-1}\mathcal{G}_i(\Delta,r)+\mathcal{V}_i(\Delta;\e),
\end{equation}
where $\mathcal{G}_i$, $i=1,2$ are given by  
\begin{equation}\label{eq:G_i}
    \mathcal{G}_1(\Delta,r):=\frac{2\sqrt2(z_2(\Delta)-r)}{\mathcal{Q}_1(\Delta)(z_2(\Delta)-z_1(\Delta))}, \;\qquad \mathcal{G}_2(\Delta,r):=\frac{2\sqrt2(r-z_1(\Delta))}{\mathcal{Q}_2(\Delta)(z_2(\Delta)-z_1(\Delta))},
\end{equation}
and $\mathcal{V}_i\in C(\Omega_0)$ for any $\e>0$. To be more precise, Lemma \ref{Lem:Rn} states that 
$R_n\in NR(\Omega_0)$ and, 
in view of \eqref{eq:Psin},  
$\mathcal{V}_i$, $i=1,2$, is given by a function depending only on $\Delta$ plus a linear combination of $\e^2R_n(\Delta;\e)$, $n=0,1$, thus it is bounded  for $\e\ll 1$.
Choose now $\delta\in\left(0,(\beta_0-\alpha_0)/2\right)$ and $\Omega_0=\Omega_0(\delta)$ such that 
\begin{equation*}
    |z_i(\Delta)-r|\geq\frac{\delta}{2}, \qquad \qquad \mbox{for any}\ \Delta\in\Omega_0, \, r\in\mathcal{R}_\delta. 
\end{equation*}
Then, by \eqref{eq:G_i}, we deduce
\begin{equation*}
    \mathcal{G}_i(\Delta,r)\geq c>0, \qquad \qquad \mbox{for any}\ \Delta\in\Omega_0, \, r\in\mathcal{R}_\delta,
\end{equation*}
and \eqref{eq:ln-pi} implies
\begin{equation}\label{eq:p_i(Delta)}
    \pi_i(\Delta)\leq \bar C\exp(-C/\e), \qquad \qquad \mbox{for } (\Delta,\e,r)\in\mbox{Sol}_\delta,
\end{equation}
where
\begin{equation*}
    \mbox{Sol}_\delta:=\left\{(\Delta,\e,r)\, :\, \Delta\in\mbox{Sol}(\e,r)\cap\Omega_0, \, \e>0, \, r\in\mathcal{R}_\delta\right\}.
\end{equation*}
On the other hand, for $K(\Delta,\e,r)=(K_1(\Delta,\e,r),K_2(\Delta,\e,r))$,  \eqref{eq:h_i} and \eqref{eq:ln-pi} yield
\begin{equation*}
    \mu_i K_i(\Delta,\e,r)=\mathcal{V}_i(\Delta;\e)-\e^{-1}\nu_i(\Delta,r), \qquad \qquad \mbox{ for } (\Delta,\e,r)\in\mbox{Sol}_\delta,
\end{equation*}
where $\nu_i(\Delta,r):=\mathcal{G}_i(\Delta,r)-c_i(r)$ and $c_i$ is defined in \eqref{eq:c1-B1}-\eqref{eq:c2-B2}.
Notice that $\mathcal{G}_i(\Delta_0,r)=c_i(r)$ and so, \eqref{eq:p_i(Delta)} implies 
\begin{equation*}
    \e^{-1}|\nu_i(\Delta,r)|\leq \bar C\exp(-C/\e), \qquad \qquad \mbox{ for } (\Delta,\e,r)\in\mbox{Sol}_\delta,
\end{equation*}
and, as a consequence, 
\begin{equation}\label{eq:est-K}
    |K(\Delta,\e,r)|\leq \Gamma, \qquad \qquad \mbox{ for } (\Delta,\e,r)\in\mbox{Sol}_\delta,
\end{equation}
for some $\Gamma>0$.

Next, we can choose $\e>0$ sufficiently small such that $(K(\Delta,\e,r),\e,r)\in\mathcal{K}$, where $\mathcal{K}$ is the neighborhood of $\mathcal{E}_0$ given by Proposition \ref{prop:pre-implicit}, namely there exists $\e_1>0$ such that $(K(\Delta,\e,r),\e,r)\in\mathcal{K}$ is compatible for any $(\Delta,\e,r)\in\mbox{Sol}_\delta$ and any $\e\in(0,\e_1)$.
Therefore, thanks to Proposition \ref{prop:pre-implicit} we end up with 
\begin{equation}\label{eq:better}
    K(\Delta,\e,r)=W(K(\Delta,\e,r),\e,r), \qquad \qquad \mbox{ for } (\Delta,\e,r)\in\mbox{Sol}_\delta,  \, \e\in(0,\e_1).
\end{equation}
From \eqref{eq:est-K} it follows that we can extract a subsequence such that
\begin{equation}\label{eq:subseq}
    \lim_{\e_j\to0^+}K(\Delta,\e_j,r)= a\in\R^2,
\end{equation}
where $a$ may depend on $r$ and the choice of the subsequence. However, from \eqref{eq:a_i}, passing to the limit in \eqref{eq:better}, we  deduce that $a$ does not depend on any admissible $\Delta\in\Omega_0$ and $r\in\mathcal{R}_\delta$, and therefore  \eqref{eq:subseq} holds for any $\e$. 
Hence,
\begin{equation*}
    K(\Delta,\e,r)\in\mathcal{B}_\delta, \qquad \qquad \mbox{ for } (\Delta,\e,r)\in\mbox{Sol}_\delta,  \, \e\in(0,\e_2),
\end{equation*}
for some $\e_2\in(0,\e_1)$.
In conclusion, we proved that if $\Delta=\Delta(\e,r)$ solves \eqref{eq:I_0,1} for $\e>0$ sufficiently small and $r\in\mathcal{R}_\delta$, then we can define $h=\pi(\Delta)$, with $\pi$ given by \eqref{eq:trans-pi} and the corresponding $k=K(\Delta,\e,r)$, obtained by \eqref{eq:h_i}, solves \eqref{eq:k-W}. However, Proposition \ref{prop:implicit} establishes uniqueness of such a solution in $\mathcal{B}_\delta$ and then we can conclude that the solution $\Delta(\e,r)$ is unique.
Since simple solutions of \eqref{eq:min} are obtained from admissible pairs $\Delta(\e,r)$ satisfying \eqref{eq:I_0,1}, the proof of Theorem \ref{thm:Maxwell} is complete.
\end{proof}

\section{Properties of the Maxwell solution and minimization of the energy}\label{sec:min}
The goal of this section is to prove that the Maxwell solution given by Theorem \ref{thm:Maxwell} minimizes the energy functional \eqref{eq:energy} with constraint \eqref{eq:constraint-intro} among all its stationary points.
In view of the change variable $x=\e t$, \eqref{eq:energy} becomes
\begin{equation*}
    E[u]=\e\int_{-\e^{-1}}^{\e^{-1}}\left[\frac{Q(\e^2u'(\e t))}{\e^2}+F(u(\e t))\right]\,dt,
\end{equation*}
so that, if we define $z(t)=u(\e t)$, we obtain the functional
\begin{equation}\label{eq:energy_eps}
    E_\e[z]=\e\int_{-\e^{-1}}^{\e^{-1}}\left[\frac{Q(\e z'(t))}{\e^2}+F(z(t))\right]\,dt.
\end{equation}
Hence, we look for a minimum of \eqref{eq:energy_eps} over the class
\begin{equation}\label{eq:zepserre}
    \mathcal{Z}_{\e r}:=\left\{z\in H^1(-\e^{-1},\e^{-1})\, : \, z>0, \; \int_{-\e^{-1}}^{\e^{-1}} z(t)\,dt=2r\e^{-1} \right\}.
\end{equation}
Let us start by computing the energy of simple solutions, which is a function depending on the admissible pair $\Delta\in\Sigma$ denoted by $E_{\e r}(\Delta)$:
\begin{equation*}
    E_{\e r}(\Delta):=\e\int_{-\e^{-1}}^{\e^{-1}}\left[\frac{Q(\e z_\Delta'(t))}{\e^2}+F(z_\Delta(t))\right]\,dt,
\end{equation*}
where $z_\Delta$ stands for the simple solution of \eqref{eq:min} corresponding to $\Delta$.
As we have seen before, $z_\Delta$ satisfies \eqref{eq:f_delta} and the definition \eqref{eq:P_eps} 
of $P_\e$ gives
\begin{equation*}
    1-\e^2 f_\Delta(z)=\frac{1}{\sqrt{1+\e^2 (z'_\Delta)^2}}.
\end{equation*}
By using the explicit formula for $Q$ \eqref{eq:Q}, we deduce 
\begin{equation*}
    Q(\e z_\Delta')=\frac{\e^2 f_\Delta(z)}{1-\e^2 f_\Delta(z)},
\end{equation*}
and, as a trivial consequence, using the definitions \eqref{eq:Gibbs} and \eqref{eq:f_delta}, one has
\begin{equation*}
    \frac{Q(\e z_\Delta')}{\e^2}+F(z_\Delta)= \frac{f_\Delta(z_\Delta)}{1-\e^2 f_\Delta(z_\Delta)}+f_\Delta(z_\Delta)+\sigma z_\Delta+b.
\end{equation*}
Substituting in the definition of $E_{\e r}(\Delta)$, we obtain
\begin{align*}
    E_{\e r}(\Delta)&=2(\sigma r+b)+\e\int_{-\e^{-1}}^{\e^{-1}}\frac{f_\Delta(z_\Delta(t))}{1-\e^2 f_\Delta(z_\Delta(t))}\,dt+\e\int_{-\e^{-1}}^{\e^{-1}}f_\Delta(z_\Delta(t))\,dt,\\
    &=2(\sigma r+b)+\e\int_{-\e^{-1}}^{\e^{-1}}\frac{2f_\Delta(z_\Delta(t))-\e^2f_\Delta(z_\Delta(t))^2}{1-\e^2 f_\Delta(z_\Delta(t))}\,dt \\
    &=2(\sigma r+b)+ \e\int_{z_1(\Delta)}^{z_2(\Delta)}\frac{2f_\Delta(s)-\e^2f_\Delta(s)^2}{1-\e^2 f_\Delta(s)}\frac{1}{H^+_\e(f_\Delta(s))}\,ds, 
\end{align*}
where in the last passage we used the change of variable $s=z_\Delta(t)$ and \eqref{eq:z'=+}.
Hence, from \eqref{eq:defP^-1} it follows that
\begin{equation}\label{eq:en-simple}
        E_{\e r}(\Delta)=2(\sigma r+b)+ \e\int_{z_1(\Delta)}^{z_2(\Delta)}\sqrt{f_\Delta(s)(2-\e^2 f_\Delta(s))}\,ds,
\end{equation}
for any $\Delta\in\mbox{Sol}(\e,r)$.
Thanks to the formula \eqref{eq:en-simple}, we can prove the following result about the energy $E_{\e r}^0$ of the Maxwell solution denoted by $z_{\e r}$:
\begin{equation*}
    E_{\e r}^0:=E_{\e}[z_{\e r}].
\end{equation*}
We recall that the Maxwell solution is the unique simple solution close to the Maxwell point $\Delta_0$  given by Theorem \ref{thm:Maxwell} for $\e\ll 1$.
\begin{prop}\label{prop:enery-Maxwell}
The energy \eqref{eq:energy_eps} of the Maxwell solution has the asymptotic form
\begin{equation}\label{eq:en-Maxw}
    E_{\e r}^0=2(\sigma_0 r+b_0)+\e c_\e+\mathcal{O}(\exp(-C/\e)),
\end{equation}
where
\begin{equation}\label{eq:c_eps}
    c_\e:=\int_{\alpha_0}^{\beta_0}\sqrt{\left[F(s)-F(\beta_0)-\sigma_0(s-\beta_0)\right]\left[2-\e^2(F(s)-F(\beta_0)-\sigma_0(s-\beta_0))\right]}\,ds.
\end{equation}
\end{prop}
\begin{proof}
By proceeding as in the proof of Lemma \ref{Lem:Rn}, we obtain that the integral function in \eqref{eq:en-simple} belongs to $NR(\Omega_0)$, that is, changing variables, $\hat E_{\e r}(h)$ is nearly regular: there exists $s<1$ such that
\begin{equation*}
    \hat E_{\e r}(h)=\hat E_{\e r}(0)+\mathcal{O}(|h|^{1-s}).
\end{equation*}
Since $h=0$ correspond to the Maxwell point $\Delta_0$, \eqref{eq:Maxwell-p} gives
\begin{align*}
    \hat E_{\e r}(0)&=2(\sigma_0 r+b_0)+\e\int_{\alpha_0}^{\beta_0}\sqrt{f_{\Delta_0}(s)[2-\e^2 f_{\Delta_0}(s)]}\,ds\\
    &=2(\sigma_0 r+b_0)+\e c_\e,
\end{align*}
where $c_\e$ is defined in \eqref{eq:c_eps} and we used \eqref{eq:Gibbs}-\eqref{eq:Maxwell-p}-\eqref{eq:f_delta}.
Therefore, we can conclude that \eqref{eq:en-Maxw} holds true thanks to exponentially smallness of $h$ given by \eqref{eq:Delta_0}.
\end{proof}
\begin{rem}
Notice that the quantity $2(\sigma_0 r+b_0)$ represents the minimum energy of \eqref{eq:energy} in the case $\e=0$, that is
\begin{equation*}
   \int_{-1}^1 F(u)\, dx
\end{equation*}
with the mass constraint \eqref{eq:constraint-intro}; see \cite{CGS}.
Moreover, the constant $c_\e$, defined in \eqref{eq:c_eps} and describing the first order correction in $\e$, satisfies
\begin{equation*}
    \lim_{\e\to0^+} c_\e=\int_{\alpha_0}^{\beta_0}\sqrt{2\left[F(s)-F(\beta_0)-\sigma_0(s-\beta_0)\right]}\,ds=:c_0,
\end{equation*}
where $c_0$ is a strictly positive constant depending only on the potential $F$.
In addition, the leading term in the difference
\begin{equation*}
     E_{\e r}^0-2(\sigma_0 r+b_0)=\e c_\e+\mathcal{O}(\exp(-C/\e))
\end{equation*}
does not depend on $r$.

Finally, if we assume that $F$ is a double well potential with wells of equal depth and $F(\alpha_0)=F(\beta_0)=0$, then $\sigma_0=b_0=0$ and
\begin{equation*}
    c_\e:=\int_{\alpha_0}^{\beta_0}\sqrt{F(s)\left[2-\e^2F(s)\right]}\,ds.
\end{equation*}
The latter constant represents the minimum energy of a transition between $\alpha_0$ and $\beta_0$ in the case of the renormalized energy studied in \cite[Proposition 2.5]{FPS} and \cite[Proposition 3.3]{FS}. 
\end{rem}

Now, we shall compare $E_{\e r}^0$ given in \eqref{eq:en-Maxw} with the energy of all other possible stationary points of \eqref{eq:energy_eps}-\eqref{eq:zepserre}, that is, solutions of \eqref{eq:min}.
Among them, let us start by observing that, given any simple solution $z:=z(t)$ of \eqref{eq:min}, its \textbf{reversal} $z^R:=z^R(t)=z(-t)$ is also a solution of \eqref{eq:min} because $Q''$ is an even function.
In particular, $(z^R)'<0$ in $(-\e^{-1},\e^{-1})$ and it is easy to check that $E_\e[z^R]=E_\e[z]$, that is the energy of a simple solution and its reversal are the same. 
Therefore, we are left to compare \eqref{eq:en-Maxw} with the energy of remaining solutions of \eqref{eq:min}, namely 
constants, (other) simple solutions, and non-monotone solutions. 
We start by analyzing the first case. 
\begin{prop}\label{prop:comp-const}
For any $\delta\in\left(0,(\beta_0-\alpha_0)/2\right)$ and $r\in\mathcal{R}_\delta$, we can choose $\e_\delta>0$ such that if $\e\in(0,\e_\delta)$, then $E^0_{\e r}$ is strictly less than the energy of constant solutions of \eqref{eq:min}.     
\end{prop}
\begin{proof}
First of all, notice that the unique constant solution of \eqref{eq:min} is exactly $z\equiv r$ and its energy \eqref{eq:energy_eps} is given by $E_\e[r]=2F(r)$.
On the other hand, from \eqref{eq:en-Maxw}, by using \eqref{eq:Maxwell-p} and \eqref{eq:Gibbs}, it follows that
\begin{align*}
    E^0_{\e r}-2F(r)&=2(\sigma_0r+\beta_0-F(r))+\mathcal{O}(\e)=2(\sigma_0r+\Phi_{\sigma_0}(\beta_0)-F(r))+\mathcal{O}(\e)\\
    &=2(\Phi_{\sigma_0}(\beta_0)-\Phi_{\sigma_0}(r))+\mathcal{O}(\e)<0,
\end{align*}
if $\e$ is sufficiently small because $\Phi_{\sigma_0}(\beta_0)$ is the global minimum of the Gibbs function $\Phi_{\sigma_0}$ and $r\in[\alpha_0+\delta,\beta_0-\delta]$. 
\end{proof}
The second step is to compare $E_{\e r}^0$ with the energy of other simple solutions, namely of simple solutions with admissible pair $\Delta$ that is not close to the Maxwell point $\Delta_0$.
We are not interested in the study of the existence of such solutions, but we can prove that if there exists a simple solution different from the Maxwell one, then $\sigma$ is bounded away from $\sigma_0$, $\Delta$ is close to $\partial_1\Sigma$ or $\partial_2\Sigma$, and the solution is close to a constant one. 
To be more precise, we have the following result.
\begin{prop}\label{prop:no-Maxw}
Choose $\rho, \delta>0$. Then, there exist $\bar\e=\bar\e(\rho,\delta)>0$ and $\rho_0=\rho_0(\delta)>0$ such that for any $\e\in(0,\bar\e)$, $r\in\mathcal{R}_\delta$ and $\Delta\in\mbox{\emph{Sol}}(\e,r)$, with $\Delta$ that is not the pair corresponding to the Maxwell solution, it follows that
\begin{align*}
    &|\sigma-\sigma_0|>\rho_0, \\
    \pi_2(\Delta)<\rho, \quad \mbox{ for } \sigma<\sigma_0, & \qquad 
    \pi_1(\Delta)<\rho, \quad \mbox{ for } \sigma>\sigma_0, \\
    |r-\beta_\sigma|<\rho, \quad \mbox{ for } \sigma<\sigma_0, & \qquad 
    |r-\alpha_\sigma|<\rho, \quad \mbox{ for } \sigma>\sigma_0. 
\end{align*}
\end{prop}
The proof of Proposition \ref{prop:no-Maxw} is very similar to the corresponding one of \cite[Theorem 7.2]{CGS} and we omit it.
It is worth observing that, as already mentioned in Remark \ref{rem:o1o2},  the continuity properties in the subregions $\Omega_1$ and $\Omega_2$ stated in \eqref{eq:InCont} of Proposition \ref{prop:In} play here a crucial role. 

By using \eqref{eq:en-simple}, \eqref{eq:en-Maxw} and Proposition \ref{prop:no-Maxw}, we prove that the energy of the Maxwell solution is strictly less than the energy of other simple solutions (if any).
\begin{prop}\label{prop:less}
For any $\delta\in\left(0,(\beta_0-\alpha_0)/2\right)$ and $r\in\mathcal{R}_\delta$, we can choose $\e_\delta>0$ such that if $\e\in(0,\e_\delta)$, then $E^0_{\e r}$ is strictly less than the energy of other simple solutions to \eqref{eq:min}.     
\end{prop}
\begin{proof}
Let $\Delta\in\mbox{Sol}(\e,r)$, with $\Delta=(\sigma,b)\neq\Delta(\e,r)$. Then, from Proposition \ref{prop:no-Maxw} we know that $\sigma \neq \sigma_0$ and we may assume that $\sigma<\sigma_0$; the case $\sigma>\sigma_0$ is similar.
For any $\delta\in\left(0,(\beta_0-\alpha_0)/2\right)$ and $r\in\mathcal{R}_\delta$,
\eqref{eq:en-simple} and \eqref{eq:en-Maxw} imply
\begin{equation*}
    E^0_{\e r}-E_{\e r}(\Delta)=2(\sigma_0 r+b_0)-2(\sigma r+b)+\mathcal{O}(\e).
\end{equation*}
By using \eqref{eq:Maxwell-p} and \eqref{eq:trans-pi}, we infer
\begin{align*}
    E^0_{\e r}-E_{\e r}(\Delta)&=2\left[(\sigma_0-\sigma)r+\Phi_{\sigma_0}(\beta_0)-\Phi_\sigma(\beta_\sigma)-\pi_2(\Delta)\right]+\mathcal{O}(\e)\\
    &=2\left[R(\sigma,r)-\pi_2(\Delta)\right]+\mathcal{O}(\e),
\end{align*}
where the $\mathcal{O}(\e)$ term is uniform for $\Delta\in\bar\Sigma$ and $r\in\mathcal{R}_\delta$.
Hence, since $\pi_2(\Delta) \geq0$ (see \eqref{eq:b}), it suffices to prove that $$R(\sigma,r)<-C_0$$ 
for any $r\in\mathcal{R}_\delta$, $\sigma<\sigma_0$, $\e\in(0,\e_\delta)$ and $\Delta\in\mbox{Sol}(\e,r)$, with $\Delta\neq\Delta(\e,r)$, for some positive constants $C_0,\e_\delta$.
Since $\alpha_0<r<z_2(\Delta)<\beta_\sigma$ (see \eqref{eq:z_i-sigma}-\eqref{eq:z_i-boundary}-\eqref{eq:r-in-z}), where $\beta_\sigma$ is a strictly increasing function of $\sigma$, we only need to study $R(\sigma,r)$ on the set
\begin{equation*}
    \mathcal{A}:=\left\{(\sigma,r)\, :\, r\in[\alpha_0,\beta_\sigma], \, \sigma\in[\underline\sigma,\sigma_0]\right\}
\end{equation*}
and prove it is uniformly negative for $\sigma \leq \sigma_0 - \tilde\rho$, for any $\tilde\rho>0$.
A direct differentiation together with \eqref{eq:punticritici} gives 
\begin{equation*}
    \frac{\partial}{\partial\sigma}R(\sigma,r)= -r+\beta_\sigma, \qquad \qquad \forall\,(\sigma,r)\in\mathcal{A},
\end{equation*}
and, since $R(\sigma_0,r)=0$ for any $r\in[\alpha_0,\beta_\sigma]$, we have
\begin{equation*}
    R(\sigma,r)<0, \qquad \qquad \forall\, \sigma<\sigma_0, \, r\in[\alpha_0,\beta_\sigma].
\end{equation*}
Now, for any $\tilde\rho>0$, there exists a positive constant $C_0$ such that
\begin{equation*}
    R(\sigma,r)\leq-C_0, \qquad 
    \mbox{ for any }  \underline\sigma<\sigma\leq\sigma_0-\tilde\rho<\sigma_0, \quad r\in[\alpha_0,\beta_\sigma].
\end{equation*}
However, Proposition \ref{prop:no-Maxw}  in the case $\sigma<\sigma_0$ implies the existence of $\e_\delta>0$ and $\rho_0>0$ such that if $\e\in(0,\e_\delta)$, $r\in\mathcal{R}_\delta$, $\Delta\in\mbox{Sol}(\e,r)$, with $\Delta$ that is not the pair corresponding to the Maxwell solution, then $\sigma<\sigma_0-\rho_0$.
Hence, we end up with
\begin{equation*}
    E^0_{\e r}-E_{\e r}(\Delta)\leq-2C_0+\mathcal{O}(\e),
\end{equation*}
and the proof is complete.
\end{proof}
It remains to compare the energy of the Maxwell solution with the one of non monotone solutions of \eqref{eq:min}.
For this, let us first say that $z\in\mathcal{Z}_{\e r}$ is a \textbf{local minimizer} of $E_\e$ if there exists a neighborhood 
 $\mathcal{N}$ of $z$ in  $\mathcal{Z}_{\e r}$ such that
\begin{equation*}
	E_\e[z]\leq E_\e[y], \qquad \qquad \forall\, y\in\mathcal{N}.
\end{equation*}
The desired property for $E_{\e r}^0$ is clearly implied by the following result.
\begin{prop}\label{prop:comp-nonmon}
Non monotone solutions of \eqref{eq:min} can not even be local minimizers of $E_\e$.
\end{prop}
\begin{proof}
Following \cite{CGS}, we shall prove that if $z$ is a non monotone solution of \eqref{eq:min}, then the second variation 
\begin{equation}\label{eq:second-nonm}
	J(z,\eta):=\frac{d^2}{d\xi^2}E_\e[z+\xi\eta]\Big|_{\xi=0}<0,
\end{equation}
at some $\eta\in H^1(-\e^{-1},\e^{-1})$ satisfying
\begin{equation}\label{eq:int-eta=0}
	\int_{-\e^{-1}}^{\e^{-1}}\eta(t)\,dt=0.
\end{equation}
Clearly, \eqref{eq:second-nonm}-\eqref{eq:int-eta=0} imply that $z$ can not be a local minimizer, because the function $z+\xi\eta$ belongs to $\mathcal{Z}_{\e r}$ for all sufficiently small $\xi\in\R$.

A direct computation, using \eqref{eq:energy_eps}, shows that
\begin{equation*}
	J(z,\eta)=\e\int_{-\e^{-1}}^{\e^{-1}}\left[Q''(\e z'(t))\eta'(t)^2 +F''(z(t))\eta(t)^2\right]\, dt.
\end{equation*}
In view of Remark \ref{rem:unique},  non monotone solutions of the problem \eqref{eq:min} are periodic solutions on the whole real line restricted in an interval larger than the period. 
In other words, there exists $T\in(-\e^{-1},\e^{-1}]$ such that
\begin{equation}\label{eq:per-instr}
	z(T)=z(-\e^{-1}), \qquad z'(T)=z'(-\e^{-1})=0.
\end{equation}
Let
\begin{equation*}
	\eta_0(t):=\begin{cases}
	z'(t), \qquad \qquad & t\in[-\e^{-1},T],\\
	0, & t\in[T,\e^{-1}],
	\end{cases}
\end{equation*}
and let $\eta_1\in H^1(-\e^{-1},\e^{-1})$ be any function satisfying \eqref{eq:int-eta=0} and
\begin{equation}\label{eq:eta1}
	\eta_1(-\e^{-1})=1, \qquad \qquad \eta_1(t)=0, \qquad \mbox{ for } t\in[T,\e^{-1}].
\end{equation}
For any $\gamma\in\R$, define $\tilde\eta_\gamma:=\eta_0+\gamma\eta_1$ and
notice that $\tilde\eta_\gamma$ satisfies \eqref{eq:int-eta=0}, so that $z+\xi\tilde\eta_\gamma\in\mathcal{Z}_{\e r}$ if $\xi$ is sufficiently small.
We have
\begin{align*}
	J(z,\tilde\eta_\gamma)&=\e\int_{-\e^{-1}}^{T}\left[Q''(\e z'(t))z''(t)^2 +F''(z(t))z'(t)^2\right]\, dt\\
	&\qquad +2\e\gamma\int_{-\e^{-1}}^{T}\left[Q''(\e z'(t))z''(t)\eta_1'(t) +F''(z(t))z'(t)\eta_1(t)\right]\, dt+\e\mathcal{O}(\gamma^2),
\end{align*}
and, integrating by parts, we obtain 
\begin{align*}
	J(z,\tilde\eta_\gamma)&=\e\int_{-\e^{-1}}^{T}\left\{-\left[Q''(\e z'(t))z''(t)\right]' +F''(z(t))z'(t)\right\}z'(t)\, dt\\
	&\qquad +2\e\gamma\int_{-\e^{-1}}^{T}\left\{-\left[Q''(\e z'(t))z''(t)\right]' +F''(z(t))z'(t)\right\}\eta_1(t)\, dt\\
	&\qquad +\e\left[-2\gamma Q''(0)z''(-\e^{-1})+\mathcal{O}(\gamma^2)\right],
\end{align*}
where we used \eqref{eq:per-instr}-\eqref{eq:eta1} to evaluate the boundary terms.
Since $Q''(0)=1$ and the ODE in \eqref{eq:min} gives
\begin{equation*}
	\left[Q''(\e z')z''\right]'=F''(z)z',\qquad \qquad \mbox{ in } (-\e^{-1},\e^{-1}),
\end{equation*}
we conclude that
\begin{equation*}
	J(z,\tilde\eta_\gamma)=\e\left[-2\gamma z''(-\e^{-1})+\mathcal{O}(\gamma^2)\right].
\end{equation*}
The solution $z$ is non monotone and, as a consequence, 
\begin{equation*}
    z''(-\e^{-1}) = \frac{F'(z(-\e^{-1})) - \sigma}{Q''(\e z'(-\e^{-1}))}\neq0.
\end{equation*}
Hence, we can choose $\gamma\in\R$ such that $J(z,\tilde\eta_\gamma)<0$, that is \eqref{eq:second-nonm} holds true and the proof is complete.
\end{proof}

Summarizing, thanks to Propositions \ref{prop:comp-const}, \ref{prop:less} and \ref{prop:comp-nonmon}, we can state that the Maxwell solution $z_{\e r}$ (and its reversal $z_{\e r}^R$) corresponding to the pair $\Delta(\e,r)$   is the stationary point of \eqref{eq:energy_eps}-\eqref{eq:zepserre} with least energy, for $r$ in closed subsets of $(\alpha_0,\beta_0)$ and $\e>0$ sufficiently small.  
As a trivial consequence, for $r$ in closed subsets of $(\alpha_0,\beta_0)$ and $\e>0$ sufficiently small, the functions
\begin{equation*}
    u_{\e r}(x):=z_{\e r}(x/\e), \qquad \qquad u_{\e r}^R(x):=u_{\e r}(-x), \qquad \qquad \qquad x\in[-1,1],
\end{equation*}
are the global minimizers in $H^1(-1,1)$ of the functional \eqref{eq:energy} with mass constraint \eqref{eq:constraint-intro}, provided the latter exists.  

We conclude our analysis by showing that the function $u_{\e r}$ (resp.\ its reversal $u_{\e r}^R$) converges, as $\e\to0$,  to the increasing (resp.\ decreasing) single interface solution defined in \eqref{eq:mineps0}, that is the minimizer  of the energy \eqref{eq:energy} for $\e=0$.
\begin{thm}\label{thm:eps-conv}
For each $r\in(\alpha_0,\beta_0)$ and $x\in[-1,1]$, one has
\begin{align*}
    \lim_{\e\to0^+} u_{\e r}(x)=u_0(x), \qquad \qquad \mbox{ if }\, x\neq-1+\ell_1,\\
    \lim_{\e\to0^+} u_{\e r}^R(x)=u_0(-x), \qquad \qquad \mbox{ if } \,x\neq-1+\ell_2,
\end{align*}
where $u_0, \ell_1$ and $\ell_2$ are defined in \eqref{eq:mineps0} and \eqref{eq:ell}.
\end{thm}
\begin{proof}
Let us denote by
\begin{align*}
    T_1(\e,r)&:=\int_{z_{\e r}(-\e^{-1})}^{z_{\e r}(-\e^{-1})+\e}\frac{ds}{H^+_\e(f_{\Delta(\e,r)}(s))},\\
    T_2(\e,r)&:=\int_{z_{\e r}(\e^{-1})-\e}^{z_{\e r}(\e^{-1})}\frac{ds}{H^+_\e(f_{\Delta(\e,r)}(s))},
\end{align*}
where $H^+_\e$ is defined in \eqref{eq:defP^-1}.
Recalling \eqref{eq:implicit-z}, $T_1(\e,r)$ represents the \emph{time} the Maxwell solution $z_{\e r}$ takes to go from its initial position $z_{\e r}(-\e^{-1})$ to $z_{\e r}(-\e^{-1})+\e$, while $T_2(\e,r)$ represents the \emph{time} to go from $z_{\e r}(\e^{-1})-\e$ to $z_{\e r}(\e^{-1})$.
We claim that for each fixed $r\in(\alpha_0,\beta_0)$
\begin{equation}\label{eq:T_i}
    T_i(\e,r)=\frac{\ell_i}{\e}+\mathcal{O}(\ln\e), \qquad \qquad \mbox{ as } \e\to0^+.
\end{equation}
Indeed, the decomposition used for \eqref{eq:splitIn}-\eqref{eq:Rn} applied to $T_1$ yields
\begin{align*}
    T_1(\e,r)=&\frac{1}{\sqrt2}\int_{z_{\e r}(-\e^{-1})}^{z_{\e r}(-\e^{-1})+\e}\frac{ds}{\sqrt{f_{\Delta(\e,r)}(s)}}\\
    &\qquad
    -\e^2\int_{z_{\e r}(-\e^{-1})}^{z_{\e r}(-\e^{-1})+\e}\frac{\sqrt{f_{\Delta(\e,r)}(s)}\left(\sqrt{2(2-\e^2f_{\Delta(\e,r)}(s))}-1\right)}{\sqrt{2(2-\e^2f_{\Delta(\e,r)}(s))}\left(\sqrt{2-\e^2f_{\Delta(\e,r)}(s)}+\sqrt2\right)}\,ds.
\end{align*}
The first integral in the previous equality is studied in \cite[Proposition 6.3]{CGS}, while the second one is a continuous function of $\Delta\in\Sigma$ and $\e>0$ sufficiently small.
Hence, we readily obtain 
\begin{align*}
T_1(\e,r)=&\sqrt2\mathcal{Q}_1(\Delta(\e,r))\left[\ln\left(\e+\lambda(\Delta(\e,r))+\sqrt{\e^2+2\lambda(\Delta(\e,r))\e}\right)-\ln\lambda(\Delta(\e,r))\right]\\
    &\qquad +J(\Delta(\e,r),\e),
\end{align*}
where $J$ is continuous on $\Omega_0\times[0,\e_0]$ for $\e_0$ sufficiently small and $\lambda$ is defined in \eqref{eq:deflambda}.
Taking advantage of \cite[Lemma 5.1]{CGS} and \eqref{eq:h_i}-\eqref{eq:Delta_0}-\eqref{eq:est-K}, we obtain
\begin{equation*}
    \lambda(\Delta(\e,r))=\mathcal{O}(\exp(-C/\e)), \qquad \qquad \ln(\lambda(\Delta(\e,r)))=-\frac{c_1(r)}{2\e}+\mathcal{O}(1).
\end{equation*}
As an obvious consequence,
\begin{equation*}
    \ln\left(\e+\lambda(\Delta(\e,r))+\sqrt{\e^2+2\lambda(\Delta(\e,r))\e}\right)=\mathcal{O(\ln\e)}.
\end{equation*}
Moreover, \eqref{eq:p-q} and \eqref{eq:Delta_0} yield
\begin{equation*}
    \mathcal{Q}_1(\Delta(\e,r))=B_1+\mathcal{O}(\exp(-C/\e)).
\end{equation*}
Putting all together and using \eqref{eq:c1-B1}, we end up with
\begin{align*}
    T_1(\e,r)=&\sqrt2 B_1\frac{c_1(r)}{2\e}+\mathcal{O}(\ln\e)=\frac{2(\beta_0-r)}{(\beta-\alpha_0)\e}+\mathcal{O}(\ln\e),
\end{align*}
that is \eqref{eq:T_i} with $i=1$ and the definition \eqref{eq:ell}.
The proof of \eqref{eq:T_i} for $i=2$ is very similar.

Now, since $u_{\e r}(x):=z_{\e r}(x/\e)$, \eqref{eq:T_i} gives
\begin{equation*}
    |z_{\e r}(-\e^{-1})-u_{\e r}(x)|\leq\e, \qquad \mbox{ for }\, -1\leq x\leq-1+\ell_1+\mathcal{O}(\e\ln\e),
\end{equation*}
and
\begin{equation*}
    |z_{\e r}(\e^{-1})-u_{\e r}(x)|\leq\e, \qquad \mbox{ for }\, 1-\ell_2+\mathcal{O}(\e\ln\e)\leq x\leq1.
\end{equation*}
Therefore, passing to the limit as $\e\to0^+$, using  \eqref{eq:O-grande} and the equality $\ell_1+\ell_2=2$, we conclude that $u_{\e r}(x)\to u_0(x)$, as $\e\to0^+$, for any $x\in[-1,1]$, with $x\neq-1+\ell_1$.
The result for $u_{\e r}^R$ is a trivial consequence and the proof is complete.
\end{proof}

\section*{Declaration of competing interests}
Authors have no competing interests to declare.

\section*{Acknowledgements}
We thank the anonymous referee for her/his comments which helped to improve the paper.
The work of R. Folino was partially supported by DGAPA-UNAM, program PAPIIT, grant IA-102423.

\end{document}